\documentclass[11pt]{article}

\usepackage{newtxtext}
\usepackage{authblk}
\usepackage{natbib}%
\usepackage[figuresright]{rotating}

\usepackage{url}
\usepackage{graphicx}
\usepackage{tikz}
\usepackage{pgfplots}
\pgfplotsset{compat=1.15}

\usetikzlibrary{positioning}

\usepackage{amsmath,amssymb,amsthm}

\usepackage{algorithm}
\usepackage{algorithmic}

\usepackage{siunitx}
\usepackage{import}

\usepackage{amsmath,amsthm}

\theoremstyle{definition}

\theoremstyle{remark}

\usepackage[margin=1in,footskip=0.25in]{geometry}
\usepackage{setspace}              % Line spacing
\usepackage{graphicx}              % \includegraphics
\usepackage{adjustbox}             % Adjust box dimensions
\usepackage{subcaption}           % Subfigures
\usepackage{caption}              % Captions customization
\usepackage{tikz}                 % TikZ graphics
\usetikzlibrary{positioning, shapes}
\usetikzlibrary{calc} 

\usepackage[table]{xcolor}
\usepackage{booktabs}
\usepackage{diagbox}
\usepackage{makecell}
\usepackage{tabularx}
\usepackage{multirow}
\usepackage{multicol}
\usepackage{threeparttable}       % For wide tables
\usepackage{nicematrix}           % Enhanced matrices
\usepackage{rotating}             % Rotated elements in tables

\usepackage{amsmath,amssymb,amsthm}
\allowdisplaybreaks[1]
\usepackage{siunitx}              % SI units
\usepackage[detect-all]{siunitx}
\usepackage{calc}                 % Width calculations

\usepackage{comment}
\usepackage{enumitem}
\usepackage{xspace}
\usepackage{textcomp}
\usepackage{csquotes}
\usepackage{ifthen}
\usepackage{soul}                 % Underlining & highlighting
\usepackage{xcolor}
\usepackage{xfp}                  % Fixed point arithmetic

\usepackage{varioref}
\usepackage{theoremref}
\usepackage[nottoc]{tocbibind}
\usepackage{lineno}
\usepackage{hyperref}
\usepackage[capitalise,noabbrev]{cleveref}

\usepackage{macros} % Your custom macros

\newtheorem{proposition}{Proposition}

\tikzset{
    dot/.style={draw,circle,inner sep=0pt,minimum size=2.5mm}
}

\title{The Quadratic Bin Packing Problem: \\ Exact Formulations and Algorithm}

\author[1]{Vítor G. Chagas\thanks{Email: \url{vitor.chagas@ic.unicamp.br} (corresponding author)}}
\author[2]{Alberto Locatelli\thanks{Email: \url{alberto.locatelli@unimore.it}}}
\author[1]{Flávio K. Miyazawa\thanks{Email: \url{fkm@ic.unicamp.br}}}
\author[2]{Manuel Iori\thanks{Email: \url{manuel.iori@unimore.it}}}

\affil[1]{\textit{Institute of Computing, University of Campinas, Campinas, SP, Brazil}}
\affil[2]{\textit{DISMI, University of Modena and Reggio Emilia, Reggio Emilia, Italy}}

\date{}

\begin{document}

\maketitle

%Abstract
\begin{abstract}
In this article, we introduce and study the Quadratic Bin Packing Problem (QBPP), which  generalizes the classical bin packing problem by introducing a fixed cost for each used bin and a pairwise cost (or profit) incurred whenever two items are packed together. Beyond its theoretical relevance, the QBPP is of practical interest due to its numerous real-world applications, mainly related to cluster analysis. To address the QBPP, we propose three compact mixed-integer linear programming (MILP) formulations, along with a set-partitioning formulation. For each compact model, we present an enhanced version with a strengthened continuous relaxation, while, for the set-partitioning formulation, we develop a tailored Branch-and-Price algorithm. Computational experiments on benchmark instances demonstrated that, while the enhanced compact formulations can be effectively solved by a standard MILP solver for small-sized instances, the Branch-and-Price approach delivered superior performance overall, especially on larger and more challenging instances.
\end{abstract}


\section{Introduction}
\label{sec:Intro}
The area of bin packing problems represents a classical and active research area in combinatorial optimization, with applications across a broad range of scientific disciplines (see, e.g., \citealt{Dyckhoff1990145}). In the classical Bin Packing Problem (BPP) (see, e.g., \citealt{kantorovich1960mathematical}), one is given a set $N = \{1,2,\dots,n\}$ of items, each characterized by a positive weight $w_i$ ($i\in N$), and an unlimited number of identical bins, each with capacity $W$. The objective is to partition the items into the minimum number of bins such that the total weight of the items assigned to any bin does not exceed its capacity. As there is a vast literature on the BPP, we refer the interested reader to the survey by \citet{delorme2016}. More recent studies published after this survey have further advanced both modeling and solution approaches, including enhanced pseudo-polynomial formulations (see, e.g., \citealt{delorme2020,deLima2023813}) and Branch-and-Price (B\&P) techniques (see, e.g., \citealt{Baldacci2024141,da2025solving}).

This article introduces and studies a generalization of the BPP, referred to as the Quadratic Bin Packing Problem (QBPP). 
In the QBPP, each bin has a fixed cost $\alpha$ associated with its use and each unordered pair of distinct items $\{i,j\}$ is associated with a dissimilarity value $d_{ij}$  ($i,j \in N$), which represents a penalty (if $d_{ij}>0$) or a profit (if $d_{ij}\leq 0$) incurred when both items are packed into the same bin. The QBPP consists in assigning each item to a bin without exceeding its capacity, with the objective of minimizing the overall cost function, given by: (i) the total cost of the bins used, plus  
(ii) the sum of the dissimilarity values associated with all pairs of items that are packed into the same bin. The problem can be modeled as the following binary quadratic program:
\begin{align}
(\mathrm{F_{QP}}) \quad \min \quad & \alpha \sum_{k \in N} y_k + \sum_{k\in N}\sum_{i\in N}\sum_{j\in N:i<j}d_{ij} x_i^k x_j^k \label{fqp_obj} \\
\text{s.t.} \quad 
& \sum_{k \in N} x_i^k = 1 && i \in N\label{fqp_constr_partition} \\
& \sum_{i \in N} w_i x_i^k \leq W y_k && k \in N \label{fqp_constr_capacity} \\
& x_i^k \in \{0,1\} && i, k \in N \label{fqp_constr_binary_x}\\
& y_k \in \{0,1\} && k \in N, \label{fqp_constr_binary_y}
\end{align}
where variable $x_i^k$ takes the value $1$ if and only if item $i$ is packed in bin $k$, while $y_k$ takes the value $1$ if and only if bin $k$ is used.  
The objective function~\eqref{fqp_obj} minimizes the overall cost.  
Constraints~\eqref{fqp_constr_partition} ensure that each item is assigned to exactly one bin, while constraints~\eqref{fqp_constr_capacity} impose that the capacity of any used bin is not exceeded. Constraints \eqref{fqp_constr_binary_x} and \eqref{fqp_constr_binary_y} define the domain of the variables. As the BPP is strongly NP-hard (see, e.g., \citealt{CoffmanGareyJohnson1996}) and arises as a special case of the QBPP when $\alpha = 1$ and $d_{ij} = 0$ ($i,j \in N$), the QBPP is also strongly NP-hard. 

The problem most closely related to the QBPP is the Quadratic Multiple Knapsack Problem (QMKP), originally proposed by~\citet{Hiley2006547}. The QMKP consists in packing items into a fixed number of knapsacks such that the total weight assigned to each knapsack does not exceed its capacity. The objective is to maximize an overall profit function given by the sum of: (i) the individual profits of the packed items, and (ii) the quadratic profit terms associated with all pairs of items packed together in the same knapsack.
Despite similarities, several key differences distinguish the QMKP from the QBPP. Unlike the QBPP, the QMKP: (i) involves a fixed number of knapsacks, (ii) does not require all items to be packed, (iii) assumes that all profits are non-negative (i.e., $d_{ij} \leq 0$; $i,j\in N$), and (iv) may allow knapsacks with different capacities, although some studies (e.g.,~\citealt{Hiley2006547}) focus on the equal-capacity case.
Over the past two decades, the QMKP has been addressed mainly through heuristic and metaheuristic approaches (see, e.g.,~\citealt{Hiley2006547,Chen2015101,Qin2016199,Galli202336}), while exact solution methods such as B\&P and Branch-and-Bound (B\&B) algorithms have emerged only in more recent years (see, e.g.,~\citealt{QMKP-bergman2019,Fleszar202289}). 

Another problem closely related to the QBPP is the Bin Packing Problem with Conflicts (BPPC), a variant of the classical BPP in which certain pairs of items are forbidden from being packed together. Introduced by \citet{1BPPC-jansen1997}, the BPPC is a special case of the QBPP in which $\alpha = 1$ and $d_{ij} = +\infty$ for each conflicting pair $\{i,j\}$ of items, and $d_{ij} = 0$ otherwise. For the BPPC, a variety of lower and upper bounding techniques have been proposed (see, e.g.,~\citealt{Gendreau2004347, Khanafer2010281, Capua2018667}). With regard to exact solution methods, different B\&P algorithms have been introduced (see, e.g.,~\citealt{1BPPC-muritiba2010, 1BPPC-elhedhli2011, 1BPPC-sadykov2013}).

Beyond its theoretical relevance, the QBPP is of practical importance due to its numerous real-world applications, particularly in Cluster Analysis (CA). According to~\cite{theodoridis2006pattern}, CA is an unsupervised learning approach aimed at identifying natural groupings of a set of objects (i.e., items), each characterized by a set of attributes that induces a measure of (dis)similarity among objects. The objective of CA is to partition the objects into disjoint clusters (i.e., bins) so that those within the same cluster are as homogeneous as possible (i.e., the total dissimilarity among objects belonging to the same clusters is minimized). CA arises in a wide range of application domains, including data analysis, graph theory, logistics, facility location, and pattern recognition, among others (see, e.g.,~\citealt{Hansen1997191,Jain1999264}). 

To maintain consistency with the bin packing terminology, from this point onward we refer to an object as an item and to a cluster as a bin. Similar to the BPP, in certain problems arising in CA, such as the node-capacitated graph partitioning problem (introduced by~\citealt{Ferreira1996247}) or capacitated clustering problems (introduced by~\citealt{Mulvey1984339}), each item is assigned a weight, and each bin is subject to a capacity constraint that limits the total weight of the items it contains.
Unlike the classical BPP, these problems involve a fixed number of bins, and the objective is to find a bin assignment that is internally homogeneous and/or well separated. Such problems can be viewed as a special case of the QBPP in which the number of used bins is fixed and $\alpha = 0$. In contrast, the objective of the BPP is to minimize the number of bins used, which, as already noted, corresponds to a special case of the QBPP where $\alpha = 1$ and $d_{ij} = 0$ ($i,j \in N$). To the best of our knowledge, the QBPP is the first problem that generalizes these problems by combining their objectives.
In this context, the bin cost $\alpha$ plays a crucial role, strongly influencing the structure of optimal QBPP solutions. Large values of $\alpha$ penalize the use of multiple bins, favoring solutions that pack items into as few bins as possible, even when highly dissimilar items are packed together. Conversely, when $\alpha$ is small, using additional bins may be advantageous, promoting solutions in which items within each bin are as similar as possible.

In this work, we propose three compact mixed integer linear programming (MILP) formulations and enhance them by means of valid inequalities and reduction techniques. We also propose a set-partitioning formulation, which we solve through a tailored B\&P algorithm. In the B\&P, the subproblem to be solved in the column generation phase generalizes the Quadratic Knapsack Problem (QKP) (see, e.g., \citealt{CILS2022_II}) by allowing both linear and quadratic terms of the objective function to be negative.  Since existing algorithms for the QKP (see, e.g., \citealt{GALLI20251}) rely on the non-negativity of these terms, we propose an efficient constructive heuristic capable of quickly generating multiple high-quality solutions for the subproblem, complemented by a tailored exact combinatorial B\&B algorithm.

The remainder of the paper is organized as follows. Section~\ref{sec:mathematical_formulations} presents the compact MILP formulations and their strengthened versions, along with the set-partitioning formulation. Section~\ref{sec:branch-and-price} describes the proposed B\&P algorithm, including the column generation procedure and both heuristic and exact approaches specifically designed for solving the pricing problem. Computational experiments and their results are reported and analyzed in Section~\ref{sec:experiments}. Finally, Section~\ref{sec:conclusion} concludes the paper and discusses directions for future research.

\section{Linear Formulations}
\label{sec:mathematical_formulations}

In this section, we present three compact MILP formulations for the $\QBPP$, called $\mathrm{F_{FGW}}$, $\mathrm{F_{2A}}$, and $\mathrm{F_{R}}$. For each formulation, we also provide an enhanced version with a strengthened continuous relaxation, called $\mathrm{F_{EFGW}}$, $\mathrm{F_{E2A}}$, and $\mathrm{F_{ER}}$, respectively. All these formulations can be solved directly using standard MILP solvers.
We then introduce a set-partitioning formulation, namely $\mathrm{F_{SP}}$, which, due to its exponential number of variables, is addressed by means of a B\&P technique described later in Section~\ref{sec:branch-and-price}.

\subsection{Fortet–Glover–Woolsey Formulation}
\label{sec:classical_linear_formulation}

The binary quadratic formulation $\mathrm{F_{QP}}$~\eqref{fqp_obj}--\eqref{fqp_constr_binary_y} can be linearized using the classical approach introduced in the early Sixties by \citet{fortet1960} for transforming algebraic functions in binary variables into linear functions. More than a decade later, the same approach was proposed, apparently independently, also by \citet{glover1974}. To apply this method to the $\QBPP$, let us replace each quadratic term~$x_i^k x_j^k$ in~\eqref{fqp_obj} with an auxiliary three-index binary variable $\hat{x}_{ij}^k$ ($i,j,k \in N$, $i < j$), which takes the value~1 if and only if both items $i$ and $j$ are packed into bin~$k$ (i.e., if only if $x_i^k x_j^k = 1$). The Fortet–Glover–Woolsey formulation ($\mathrm{F_{FGW}}$) is obtained by introducing in $\mathrm{F_{QP}}$, for each auxiliary variable $\hat{x}_{ij}^k$, the corresponding linear linking constraints, as follows:
\begin{align} (\mathrm{F_{FGW}}) \quad \min \quad & \alpha \sum_{k \in N} y_k + \sum_{k\in N}\sum_{i\in N}\sum_{j\in N:i<j} d_{ij} \hat{x}_{ij}^k \label{ffgw_obj} \\ 
\text{s.t.} \quad & \eqref{fqp_constr_partition}-\eqref{fqp_constr_binary_y} \nonumber\\ 
& \hat{x}_{ij}^k \leq x_i^k && i,j,k \in N, \ i<j \label{ffgw_constr_1} \\ 
& \hat{x}_{ij}^k \leq x_j^k && i,j,k \in N, \ i<j \label{ffgw_constr_2} \\ 
& x_i^k + x_j^k - \hat{x}_{ij}^k \leq 1 && i,j,k \in N, \ i<j \label{ffgw_constr_3} \\ 
&  \hat{x}_{ij}^k \in \{0,1\} && i,j,k \in N, \ i<j.\label{ffgw_constr_bounds} \end{align}
Constraints \eqref{ffgw_constr_1} and \eqref{ffgw_constr_2} ensure that $\hat{x}_{ij}^k$ takes the value~0 if $x_i^k x_j^k = 0$, while constraints \eqref{ffgw_constr_3} force $\hat{x}_{ij}^k$  to take the value~1 if both $x_i^k$ and $x_j^k$ are equal to~1. 

The applied linearization technique increases the size of $\mathrm{F_{QP}}$ by introducing $\frac{n^2(n-1)}{2}$ additional variables and $2n^2(n-1)$ additional constraints. However, some simplifications in $\mathrm{F_{FGW}}$ are possible. First, all auxiliary variables $\hat{x}_{ij}^k$ with $d_{ij}=0$, together with their linking constraints, can be removed. Moreover, the variable domain constraints \eqref{ffgw_constr_bounds} can be replaced by $\hat{x}_{ij}^k \ge 0$ ($i,j,k \in N$, $i<j$). Additional reductions can be obtained by eliminating redundant constraints. In particular, for any $i,j,k \in N$ with $i<j$, constraints \eqref{ffgw_constr_3} and  \eqref{ffgw_constr_bounds} are redundant if $d_{ij}<0$, whereas constraints~\eqref{ffgw_constr_1} and~\eqref{ffgw_constr_2} are redundant if $d_{ij}>0$. 

These simplifications follow from the observation that if $d_{ij}\le 0$, then in any optimal solution $(x^*,y^*,\hat{x}^*)$ of $\mathrm{F_{FGW}}$, each variable $\hat{x}_{ij}^{*k}$ takes the largest value allowed by \eqref{ffgw_constr_1} and \eqref{ffgw_constr_2}, namely $\hat{x}_{ij}^{*k}=\min\{x_i^{*k}, x_j^{*k}\}\in\{0,1\}$. Consequently, $\hat{x}_{ij}^{*k}$ automatically satisfies constraints \eqref{ffgw_constr_3} and \eqref{ffgw_constr_bounds}.
On the other hand, if $d_{ij}\geq 0$, each variable $\hat{x}_{ij}^{*k}$ takes the smallest non-negative value allowed by \eqref{ffgw_constr_3}, i.e., $\hat{x}_{ij}^{*k}=\max\{0, x_i^{*k} + x_j^{*k} -1 \}=\min\{x_i^{*k}, x_j^{*k}\} \in \{0,1\}$. As a result, $\hat{x}_{ij}^{*k}$ is binary and satisfies constraints  \eqref{ffgw_constr_1} and \eqref{ffgw_constr_2}.

Moreover, formulation $\mathrm{F_{FGW}}$ admits a large number of symmetric solutions. Indeed, given any feasible solution of $\mathrm{F_{FGW}}$, any permutation of the bins leads to another feasible solution with the same objective value. To reduce this redundancy, let us replace variable $y_k$ with $x_k^k$ ($k\in N$) and introduce the following  symmetry-breaking inequalities surveyed by \citet{delorme2016}:
\begin{align}
& x_i^k = 0 && i, k \in N,\ i <k \label{eq:constr_x_symm1} \\
& x_j^i \le x_i^i && i, j \in N,\ i< j . \label{eq:constr_x_symm2}
\end{align}
Constraints~\eqref{eq:constr_x_symm1} force item $i$ to be packed only in the first $i$ bins, while constraints~\eqref{eq:constr_x_symm2} ensure that items with index greater than $i$ can be placed in bin $i$ only if item $i$ is also packed in bin $i$.
The same symmetry-breaking approach has been used for other BPPs, including the ordered open-end BPP (see, e.g., \citealt{ceselli2008}) and the BPP with fragility objects  (see, e.g., \citealt{CLAUTIAUX201473,daSilva2025}).
\subsection{Extended Fortet–Glover–Woolsey Formulation}
\label{sec:reformulation_linearization}

To strengthen the continuous relaxation of $\mathrm{F_{FGW}}$, we adopt the classical Reformulation Linearization Technique (RLT), originally introduced by \citet{adams1986} for quadratic binary programs and later extended to general binary programs by \citet{adams1990}.
In order to apply RLT, we introduce into $\mathrm{F_{FGW}}$ additional quadratic constraints obtained by multiplying the original capacity constraints \eqref{fqp_constr_capacity} by the variable $x_j^k$ and by its complement $1 - x_j^k$ ($j,k \in N$). Noting that the square of a binary variable equals the variable itself, we obtain the following two families of quadratic constraints:
\begin{align}
&\sum_{i\in N:i\neq j} w_i x_{i}^{k}x_j^k \leq W x_k^k x_j^k  - w_j x_{j}^{k}
&& j,k\in N \label{RL_constr_capacity_quad_1}\\
&\sum_{i \in N : i \neq j} w_i (x_i^k - x_{i}^k x_{j}^k) \leq W(x_k^k - {x}_{k}^k {x}_{j}^k)
&& j,k\in N.\label{RL_constr_capacity_quad_2}
\end{align}
To linearize constraints \eqref{RL_constr_capacity_quad_1} and \eqref{RL_constr_capacity_quad_2}, we introduce the binary variables $\tilde{x}_{ij}^k$ ($i,j,k \in N$, $i \neq j$). These variables extend the variables $\hat{x}_{ij}^k$ ($i,j,k \in N, \ i<j$) to all ordered pairs $(i,j)$ with $i \neq j$. Specifically, $\tilde{x}_{ij}^k$ is equal to $x_i^k x_j^k$ ($i,j,k \in N$, $i \neq j$) and therefore takes the value 1 if and only if both items $i$ and $j$ are packed into bin~$k$. The resulting $\mathrm{F_{EFGW}}$ formulation is then:

\begin{align} (\mathrm{F_{EFGW}}) \quad \min \quad & \alpha \sum_{k \in N} x_k^k + \sum_{k\in N}\sum_{i\in N}\sum_{j\in N:i<j} d_{ij} \tilde{x}_{ij}^k \label{rl_obj} \\ \text{s.t.} \quad & \eqref{fqp_constr_partition}-\eqref{fqp_constr_binary_x}, \eqref{eq:constr_x_symm1},  \eqref{eq:constr_x_symm2} \nonumber\\ 
&\sum_{i\in N:i\neq j} w_i \tilde{x}_{ij}^{k}\leq W \tilde{x}_{kj}^k - w_j x_{j}^{k}
&& j,k\in N \label{RL_constr_capacity_lin_1}\\
&\sum_{i \in N : i \neq j} w_i (x_i^k - \tilde{x}_{ij}^k) \leq W(x_k^k - \tilde{x}_{kj}^k)
&& j,k\in N\label{RL_constr_capacity_lin_2}\\
& \tilde{x}_{ij}^k \leq x_i^k && i,k \in N, j\in N\setminus\{i\} , d_{ij}<0 \label{rl_constr_1} \\
& x_i^k + x_j^k - \tilde{x}_{ij}^k \leq 1 && i,k \in N, j\in N\setminus\{i\} , d_{ij}>0  \label{rl_constr_3} \\ 
& \tilde{x}_{ij}^k =\tilde{x}_{ji}^k  && i,j,k \in N, \ i < j \label{rl_constr_4} \\
&  \tilde{x}_{ij}^k \geq 0 && i,k \in N, j\in N\setminus\{i\}, d_{ij}\neq 0.\label{rl_constr_bounds} \end{align}
Constraints \eqref{RL_constr_capacity_lin_1} and \eqref{RL_constr_capacity_lin_2} are obtained by replacing each quadratic term $x_i^k x_j^k$ in \eqref{RL_constr_capacity_quad_1} and \eqref{RL_constr_capacity_quad_2}, respectively,  with the auxiliary binary variable $\tilde{x}_{ij}^k$ ($i,j,k \in N$, $i \neq j$), while constraints \eqref{rl_constr_1}-\eqref{rl_constr_bounds} are the linking and domain constraints for the auxiliary binary variable $\tilde{x}_{ij}^k$.

\subsection{Two-index Assignment Formulation}
\label{sec:two_index_assignment_formulation}

Models $\mathrm{F_{FGW}}$ and $\mathrm{F_{EFGW}}$ have $\bigO(n^3)$ variables and constraints. A more compact linearization of $\mathrm{F_{QP}}$, with $\bigO(n^2)$ variables and $\bigO(n^3)$ constraints, can be obtained by following a classical approach used for general partitioning problems  (see, e.g., \citealt{fan2010}).
The idea is to retain the item–bin assignment variables $x_j^k$ ($j,k \in N$) of $\mathrm{F_{QP}}$  and  replace each quadratic term $\sum_{k \in N} x_i^k x_j^k$ in \eqref{fqp_obj} with an auxiliary two-index binary variable $z_{ij}$ ($i,j \in N$, $i < j$), which takes the value 1 if and only if items $i$ and $j$ are packed in the same bin  (i.e., $z_{ij} = 1$ if and only if $\sum_{k \in N} x_i^k x_j^k = 1$).  The resulting Two-index Assignment Formulation  ($\mathrm{F_{2A}}$) is then:
\begin{align} (\mathrm{F_{2A}}) \quad \min \quad & \alpha \sum_{k \in N} y_k + \sum_{i \in N}\sum_{j \in N: i<j} d_{ij} \, z_{ij} \label{fnc_obj} \\
\text{s.t.} \quad & \eqref{fqp_constr_partition} - \eqref{fqp_constr_binary_y} \nonumber \\
 & z_{ij} + x_i^k - x_j^k \le 1 && i,j,k \in N,\ i<j \label{fnc_constr_1} \\
& z_{ij} - x_i^k + x_j^k \le 1 && i,j,k \in N,\ i<j \label{fnc_constr_2} \\
& -z_{ij} + x_i^k + x_j^k \le 1 && i,j,k \in N,\ i<j \label{fnc_constr_3} \\
&  z_{ij} \in \{0,1\} && i,j \in N,\ i<j. \label{fnc_constr_bounds_z} \end{align}
Constraints \eqref{fnc_constr_3} ensure that $z_{ij}$ takes the value 1 if items $i$ and $j$ are packed in the same bin, while constraints \eqref{fnc_constr_1} and \eqref{fnc_constr_2} ensure that $z_{ij}$ takes the value 0 otherwise.

Following the same reasoning as for $\mathrm{F_{FGW}}$, some simplifications can be applied to $\mathrm{F_{2A}}$. Specifically, all variables $z_{ij}$ with $d_{ij}=0$, together with their linking constraints, can be removed; the domain constraints \eqref{fnc_constr_bounds_z} can be replaced by $z_{ij} \ge 0$; constraints \eqref{fnc_constr_3} and \eqref{fnc_constr_bounds_z} when $d_{ij}<0$, as well as constraints \eqref{fnc_constr_1} and \eqref{fnc_constr_2} when $d_{ij}>0$, can be removed. Furthermore, as discussed in Section \ref{sec:classical_linear_formulation}, we can replace variable $y_k$ with $x_{k}^k$ ($k \in N$) and introduce in $\mathrm{F_{2A}}$ the symmetry-breaking inequalities \eqref{eq:constr_x_symm1} and \eqref{eq:constr_x_symm2} to reduce the large number of equivalent solutions.

\subsection{Extended Two-index Assignment Formulation}
\label{sec:aggregated_reformulation_linearization}

To strengthen the continuous relaxation of $\mathrm{F_{2A}}$ while keeping $\bigO(n^2)$ variables, we aggregate constraints \eqref{RL_constr_capacity_quad_1} and \eqref{RL_constr_capacity_quad_2}, obtaining, respectively, the following two families of surrogate quadratic constraints:
\begin{align}
&\sum_{k \in N}\sum_{i\in N:i\neq j} w_i x_{i}^{k}x_j^k \leq W \sum_{k \in N}x_k^k x_j^k  - w_j \sum_{k \in N}x_{j}^{k}
&& j\in N \label{RL_constr_capacity_quad_1_surr}\\
&\sum_{k \in N} \sum_{i \in N : i \neq j} w_i (x_i^k - x_{i}^k x_{j}^k) \leq W\sum_{k \in N} (x_k^k - {x}_{k}^k {x}_{j}^k)
&& j\in N.\label{RL_constr_capacity_quad_2_surr}
\end{align}
By replacing each quadratic term $x_i^k x_j^k$ in \eqref{RL_constr_capacity_quad_1_surr} and \eqref{RL_constr_capacity_quad_2_surr} with $\tilde{x}_{ij}^k$, and using the identities $\sum_{k \in N} \tilde{x}_{ij}^k = z_{ij}$ and $\sum_{k \in N} x_j^k = 1$, we can linearize \eqref{RL_constr_capacity_quad_1_surr} and \eqref{RL_constr_capacity_quad_2_surr} as:
\begin{align}
&\sum_{i \in N : i \neq j} w_i z_{ij}
    \leq W \sum_{k \in N} \tilde{x}_{kj}^k - w_j
&& j \in N \label{eq:constr_rl_lin_aggr1}\\
&\sum_{i \in N : i \neq j} w_i(1- z_{ij})\leq W \sum_{k \in N} (x_k^k - \tilde{x}_{kj}^k)
&& j \in N, \label{eq:constr_rl_lin_aggr2}
\end{align}
respectively.
The $\mathrm{F_{E2A}}$ formulation is obtained from $\mathrm{F_{2A}}$ by applying the simplifications described in Section~\ref{sec:two_index_assignment_formulation} and by adding inequalities~\eqref{eq:constr_rl_lin_aggr1} and~\eqref{eq:constr_rl_lin_aggr2}. As these inequalities involve the auxiliary two-index variables $\tilde{x}_{kj}^k$, such variables must be introduced together with their associated linear linking constraints, as follows:
\begin{align} (\mathrm{F_{E2A}}) \quad \min \quad & \alpha \sum_{k \in N} x_k^k + \sum_{i \in N}\sum_{j \in N: i<j} d_{ij} \, z_{ij} \label{arl_obj} \\
\text{s.t.} \quad & \eqref{fqp_constr_partition}-\eqref{fqp_constr_binary_x}, \eqref{eq:constr_x_symm1},  \eqref{eq:constr_x_symm2}, \eqref{eq:constr_rl_lin_aggr1}- \eqref{eq:constr_rl_lin_aggr2} \nonumber \\
& z_{ij} + x_i^k - x_j^k \le 1 && i,j,k \in N,\ i<j, d_{ij}<0  \label{arl_constr_1} \\
& z_{ij} - x_i^k + x_j^k \le 1 && i,j,k \in N,\ i<j, d_{ij}<0  \label{arl_constr_2} \\
& -z_{ij} + x_i^k + x_j^k \le 1 && i,j,k \in N,\ i<j, d_{ij}>0  \label{arl_constr_3} \\
& \tilde{x}_{kj}^k \leq x_k^k && j,k \in N, d_{kj}<0 \label{arl_constr_4} \\
& x_k^k + x_j^k - \tilde{x}_{kj}^k \leq 1 && j,k \in N, d_{kj}>0  \label{arl_constr_5} \\ 
& \tilde{x}_{kj}^k =\tilde{x}_{jk}^k  && j,k \in N \label{arl_constr_6} \\
&  \tilde{x}_{kj}^k \geq 0 && j,k \in N\\
&  z_{ij} \geq 0&& i,j \in N,\ i<j, d_{ij}\neq 0.
\end{align}

\subsection{Representative Formulation}
\label{sec:edge-representative_formulation}

Another compact linear formulation, requiring only $\bigO(n^2)$ variables, can be obtained by introducing representative variables $r_j$ instead of the item–bin assignment variables $x_j^k$ ($j,k \in N$) used in previous formulations.  
This approach has been successfully applied in the literature on different graph partitioning problems (see, e.g., \citealt{Campelo2008}). The variable $r_j$ takes the value $1$ if and only if item $j$ is the item of smallest index in the bin in which it is packed. The resulting Representative Formulation ($\mathrm{F_{R}}$) for the QBPP reads as follows:
\begin{align}
 (\mathrm{F_{R}}) \quad \min \quad & \alpha \sum_{k \in N} r_k + \sum_{i \in N}\sum_{j \in N: i<j} d_{ij} \, z_{ij} \label{fer_obj} \\
\text{s.t.} \quad 
   &\sum_{j \in N: i<j} w_jz_{ij} \leq W - w_i&& i\in N\label{eq:fer_constr_capacity_repr}\\
& z_{ij} + z_{ik} - z_{jk} \leq 1 
&& i\in N, j,k \in N, i<j<k  \label{eq:constr_triangle1} \\
& z_{ij}  + z_{jk} - z_{ik} \leq 1 
&& i\in N, j,k \in N, i<j<k  \label{eq:constr_triangle2} \\
&  z_{ik} + z_{jk}-z_{ij}  \leq 1 
&& i\in N, j,k \in N, i<j<k  \label{eq:constr_triangle3} \\
& r_j + z_{ij} \leq 1 
&& i,j \in N,i<j \label{eq:fer_constr_repr_leq} \\
& r_j + \sum_{i \in N: i<j} z_{ij} \geq 1 
&& j \in N \label{eq:fer_constr_repr_geq} \\
& r_j \in\{0, 1\} 
&& j \in N \label{eq:fer_constr_int_repr}\\
& z_{ij} \in \{0,1\}
&& i,j \in N, i<j. \label{eq:fer_constr_int_z}
\end{align}
Constraints~\eqref{eq:fer_constr_capacity_repr} impose that the capacity of any used bin is not exceeded.  
Triangle inequalities~\eqref{eq:constr_triangle1}-\eqref{eq:constr_triangle3} guarantee that, for any three distinct items $i,j,k \in N$, if items $i$ and $j$ are packed in the same bin and items $i$ and $k$ are packed in the same bin, then items $j$ and $k$ must also be packed in the same bin. Constraints~\eqref{eq:fer_constr_repr_leq} and~\eqref{eq:fer_constr_repr_geq} ensure that item $j$ is the representative of the bin in which it is packed (i.e., $r_j=1$) if and only if there is no item $i$ packed in that bin such that $i<j$.

\subsection{Extended Representative Formulation}
\label{sec:representative_reformulation_linearization}

To strengthen the continuous relaxation of $\mathrm{F_{R}}$, we introduce the following family of additional quadratic constraints:
\begin{align}
& r_j + \sum_{i \in N: i<j} r_i z_{ij} = 1 && j \in N. \label{eq:constr_erl_quadratic_3}
\end{align}
Constraints~\eqref{eq:constr_erl_quadratic_3} ensure that each item is either the representative item of its bin or it is assigned to a bin whose representative item has a smaller index.
To linearize these quadratic constraints, we introduce auxiliary binary variables $\hat{z}_{ij} = r_i z_{ij}$ ($i,j \in N$,  $i<j$). The resulting $\mathrm{F_{ER}}$ formulation reads as follows:
\begin{align}
(\mathrm{F_{ER}}) \quad 
& \eqref{fer_obj}-\eqref{eq:fer_constr_int_z}  \nonumber \\
& r_j + \sum_{i \in N: i<j} \hat{z}_{ij} = 1 && j \in N\label{eq:constr_rrl_linear_2}\\
& \hat{z}_{ij} \le r_i && i,j \in N, i<j \label{fer_rl_constr_link1} \\
& \hat{z}_{ij} \le z_{ij} && i,j \in N, i<j \label{fer_rl_constr_link2} \\
& r_i + z_{ij} - \hat{z}_{ij} \le 1 && i,j \in N, i<j \label{fer_rl_constr_link3} \\
& 0 \le \hat{z}_{ij} \le 1 && i,j \in N, i<j,\label{fer_rl_constr_bounds}
\end{align}
where constraints~\eqref{eq:constr_rrl_linear_2} are the linearized counterparts of constraints~\eqref{eq:constr_erl_quadratic_3}, respectively, while constraints~\eqref{fer_rl_constr_link1}–\eqref{fer_rl_constr_bounds} are the linking and domain constraints for the introduced auxiliary variables $\hat{z}_{ij}$.

Observe that the role of the variables $x_j^j$ and $x_{ij}^i$, in $\mathrm{F_{E2A}}$, coincides with the role of the variables $r_j$ and $\hat{z}_{ij}$, respectively, in $\mathrm{F_{ER}}$. Indeed, if the variable $x_j^j$ takes the value 1, then item $j$ is the representative item of the bin in which it is packed (i.e., $r_j = 1$). Moreover, if the variable $x_{ij}^i$ takes value 1, then item $j$ is packed in a bin whose representative item is $i$ (i.e., $\hat{z}_{ij} = 1$).

\subsection{Set-Partitioning Formulation}
\label{sec:set_partitioning_formulation}

The Set-Partitioning formulation involves an exponential number of variables, each associated with a {\em pattern}, i.e., a subset of items whose total weight does not exceed the bin capacity~$W$. Let $\mathcal{P} = \{P \subseteq N : \sum_{i \in P} w_i \leq W\}$ denote the set of all feasible patterns. For each pattern $P \in \mathcal{P}$, we define its cost $c_P = \alpha + \sum_{i \in P} \sum_{j \in P : i < j} d_{ij}$, and introduce a binary variable $\lambda_P$, taking the value $1$ if and only if pattern $P$ is selected in the solution. The Set-Partitioning Formulation ($\mathrm{F_{SP}}$) reads as follows:
\begin{align}
(\mathrm{F_{SP}})  \quad \min \quad & \sum_{P \in \mathcal{P}} c_P \lambda_P \label{eq:fbp_obj} \\
\text{s.t.} \quad
& \sum_{P \in \mathcal{P}:i\in P} \lambda_P = 1 &&  i \in N \label{eq:fbp_constr_assigned_item} \\
& \lambda_P \in \{0,1\} &&  P \in \mathcal{P}. \label{eq:fbp_constr_int_lambda}
\end{align}
The objective function~\eqref{eq:fbp_obj} minimizes the total cost of the used patterns, while constraints~\eqref{eq:fbp_constr_assigned_item} ensure that all items are packed exactly once.
Note that if all the dissimilarity values are non-negative (i.e., $d_{ij}\geq 0$, $i,j\in N$), constraints~\eqref{eq:fbp_constr_assigned_item} can be replaced by $\sum_{P \in \mathcal{P}:i\in P} \lambda_P \geq 1$ ($i \in N$), thus leading to a set-covering formulation.

\section{Branch-and-Price Algorithm}
\label{sec:branch-and-price}
In the formulation $\mathrm{F_{SP}}$, the number of variables grows exponentially with respect to~$n$. In practice, $\mathrm{F_{SP}}$ cannot be solved directly using approaches based on the exhaustive enumeration of all patterns (also defined columns from now on). Therefore, in this section, we present a new B\&P algorithm for its solution. The B\&P solves a Master Problem (MP), which is defined as the linear relaxation of $\mathrm{F_{SP}}$ in which the integrality constraints~\eqref{eq:fbp_constr_int_lambda} are replaced by
$\lambda_P \geq 0$ ($ P \in \mathcal{P}$).

At the root node of the B\&P tree, we initialize the Restricted Master Problem (RMP) by selecting, from MP, only the columns corresponding to single-item patterns. Once the RMP is solved to optimality, the resulting dual solution is used to define the pricing problem and search for columns with negative reduced cost (see Section~\ref{sec:pricing_problem}).
As the pricing problem corresponds to a generalization of the classical Quadratic Knapsack Problem ($\QKP$), which is strongly $\classNP$-hard (see, e.g., \citealt{CILS2022_II}), to solve the problem efficiently, we first apply a heuristic algorithm (see Section~\ref{sec:QKP_heuristic}). If no column with negative reduced cost is found, an exact combinatorial algorithm, presented in Section~\ref{sec:QKP_exact}, is then used to solve the pricing problem to optimality. If a column with negative reduced cost is found, it is added to the RMP, and the column generation process is repeated. Otherwise, the current RMP solution $\lambda^*_P$ is optimal for the MP, and its objective value provides a valid lower bound for the QBPP. 
The search tree is then explored according to the branching scheme described in Section~\ref{sec:branching_rules}, and the branching decisions are imposed in the pricing problem of the generated nodes.

\subsection{Column Generation}
\label{sec:pricing_problem}

Solving the RMP yields an optimal solution with respect to the current set of columns, which is in general suboptimal for the full MP. To verify optimality or otherwise identify improving columns, we consider the dual of the MP, given by:
\begin{equation}\label{eq:MP_dual_formulation}
\max \left\{
\sum_{i \in N} \pi_i :\; \sum_{i \in P} \pi_i \leq c_P, \;P \in \mathcal{P}, \; \pi_i \gtreqless 0, \;i \in N
\right\},
\end{equation}
where $\pi_i$ is the dual variable corresponding to the $i$-th  constraint in~\eqref{eq:fbp_constr_assigned_item}. 
Once the RMP is solved to optimality, the corresponding dual solution $\pi^*$ may violate a dual constraint in~\eqref{eq:MP_dual_formulation}, inducing a column with negative reduced cost in the MP. If such a violated dual constraint exists, the column generation (CG) procedure adds the corresponding column to the RMP and re-optimizes it. This process is repeated until no further violated dual constraints occur, i.e., until the current RMP optimal solution is also optimal for the MP.
At each CG iteration, a separation problem, called pricing problem, is solved in order to determine (if any) a pattern $P^*\in\mathcal{P}$ such that $\sum_{i \in P^*} \pi_i^* > c_P$.
Let us denote by $p_{ij} = -d_{ij}$ and
let $z_i$ be a binary variable that takes the value $1$ if and only if item $i\in N$ is in the pattern $P$.
As $c_P = \alpha + \sum_{i \in P} \sum_{j \in P : i < j} d_{ij}$, the pricing problem asks to solve the following Generalized Quadratic Knapsack Problem ($\generalQKP$):
\begin{equation} \label{eq:QKP_formulation}
    z(\pi^* )=\max \left\{ \sum_{i \in N} \pi_i^* z_i + \sum_{i \in N} \sum_{j \in N : i < j} p_{ij} z_i z_j:\; \sum_{i \in N} w_i z_i\leq W,\; z_i\in\{0,1\},\; i \in N \right\}.
\end{equation}
If $z(\pi^*) > \alpha$, a violating pattern $P^* \in \mathcal{P}$, corresponding to an optimal solution of~\eqref{eq:QKP_formulation}, has been found. Thus, variable $\lambda_{P^*}$ has a negative reduced cost and can be added to the RMP. Otherwise, the RMP solution is also optimal for the MP. 

Observe that both $\pi^*_i$ and $p_{ij}$ in \eqref{eq:QKP_formulation} may be  negative, whereas in the literature of the standard $\QKP$ the linear and quadratic terms are typically assumed to be non-negative (see, e.g., \citealt{caprara1999}).
This generalization is crucial, as classical algorithms for the $\QKP$ (see, e.g., \citealt{GALLI20251}) rely on the non-negativity of these terms, and therefore cannot be applied directly to the $\generalQKP$.

\subsubsection{Multiple Constructive Heuristic  for the GQKP} % $\generalQKP$
\label{sec:QKP_heuristic}

This section presents the Multiple Constructive Heuristic (MCH) that we designed to generate multiple high-quality solutions for the pricing problem~\eqref{eq:QKP_formulation}.
The MCH extends the Constructive Heuristic (CH) proposed by \citet{QKP-fomeni2014} for the $\QKP$ by maintaining multiple candidate solutions at each dynamic programming state.

The original CH approach is based on the classical dynamic programming (DP) recursion for the linear Knapsack Problem (KP01), preserving its stages and states while replacing the linear objective function with the quadratic objective function of the $\QKP$.
This approach applied to~\eqref{eq:QKP_formulation} works as follows.
For each $i \in N$ and $w \in \srange{0}{W}$, let $f(i,w)$ denote the profit of the best-found pattern $P(i,w) \subseteq \srange{1}{i}$ such that $\sum_{j \in P(i,w)} w_j = w$. 
With base cases $f(0,0) = 0$ and $f(0,w) = -\infty$ for $w \in \srange{1}{W}$, the DP recursion for the $\generalQKP$ is then defined as:

% \scalebox{0.94}{
\begin{equation}\nonumber
f(i,w) =
\begin{cases}
\max \left\{ f(i-1,w), f(i-1,w-w_i) + p\left(i,\,P(i-1,w-w_i)\right) \right\} & \text{if } w_i \le w, \\
f(i-1,w) & \text{otherwise},
\end{cases}
\end{equation}
% }
%
where $p(i,P(i-1,w-w_i)) = \pi^*_i + \sum_{j \in P(i-1,w-w_i)} p_{ij}$ denotes the marginal quadratic contribution incurred by inserting item $i$ into pattern $P(i-1,w-w_i)$. Since the evaluation of each state $f(i,w)$ requires $\mathcal{O}(n)$ time and there are $nW$ states in total, the overall time complexity of the algorithm is $\mathcal{O}(n^2W)$. 

The proposed MCH generalizes this recursion by retaining multiple solutions per state. Specifically, instead of storing only the best pattern found for each state $(i,w)$, the MCH keeps the $h$ best patterns found so far.
Formally, for $i \in N$,  $w \in \{0,\dots,W\}$, and a given parameter $h$, let $\mathcal{P}^h(i,w)=\{P^j(i,w)\}_{j=1}^{h}$ 
denote the set of the $h$ best patterns found considering the first $i$ items and capacity $w$, and let $\mathcal{F}^h(i,w)=\{f^j(i,w)\}_{j=1}^{h}$ be the corresponding set of values.
Given a finite totally ordered set $(S,\leq)$ and an integer $h \in \{1, \dots, |S|\}$, let $\operatorname{top}^h(S)$ denote the set of the $h$ largest elements of $S$.  When item $i$ is considered for insertion into the patterns in $\mathcal{P}^h(i-1, w-w_i)$, the set $\mathcal{P}^h(i, w)$ is updated by selecting the first $h$ patterns according to the corresponding values in $\mathcal{F}^h(i, w)$, which are computed as follows:

% \scalebox{0.94}{
\begin{equation}\nonumber
\mathcal{F}^h(i,w) =
\begin{cases}
\operatorname{top}^h\left( \mathcal{F}^h(i-1,w)\cup \left\{f^j(i-1,\,w-w_i) + p\left(i,\,P^j(i-1,\,w-w_i)\right)\right\}_{j=1}^h \right) & \text{if } w_i \le w \\
\mathcal{F}^h(i-1,w) & \text{otherwise}.
\end{cases}
\end{equation}
% }

Regarding the time complexity of the MCH, first note that obtaining the candidate values for $\mathcal{F}^h(i,w)$ takes time $\bigO(nh)$. In addition, as the number of candidates is at most $2h$, sorting them requires $\bigO(h \log h)$ time. Thus, obtaining $\mathcal{F}^h(i,w)$ takes time $\bigO(nh + h \log h)$.
Since MCH computes $\mathcal{F}^h(i,w)$ for every pair $(i, w)$, the total time complexity of MCH is then $\bigO(hnW(n+\log(h)))$.

Moreover, as the initial order of the items in $N$ may impact the quality of the obtained solutions (see, e.g., \citealt{QKP-fomeni2014}), based on preliminary experiments, we found that sorting the items in non-increasing order of $(p_i + \sum_{j \neq i} p_{ij})/w_i$ yields the best overall performance.

The MCH offers a twofold advantage over the CH: (i) it potentially enhances the quality of the solution found, and (ii) it enables the generation of multiple columns with negative reduced cost.
As shown below in Section~\ref{sec:experiments}, these two characteristics contribute to improving the performance of the B\&P algorithm.
 
\subsubsection{Exact Algorithm for the GQKP} % $\generalQKP$
\label{sec:QKP_exact}

In case the heuristic fails to obtain a column with negative reduced cost, the pricing problem must be solved exactly in order to either identify an improving column or certify that none exists. To this end, we propose a combinatorial B\&B algorithm to solve the pricing problem~\eqref{eq:QKP_formulation} to optimality.
Each node $v = (I,O,U)$ of the B\&B tree represents a subproblem corresponding to a partial knapsack filling, where $N$ is partitioned into three disjoint subsets: $I$, $O$, and $U$. Specifically, $I$ and $O$ denote the sets of items for which a decision has already been made, with $I$ containing the items packed into the knapsack and $O$ those excluded,  whereas $U$ contains the undecided items that can still be packed into the knapsack. The value of the incumbent solution $z^*$ is initialized using the value of the best solution found by the MCH computed in the current CG iteration.

\paragraph{Bounding functions.}
For each visited node $v = (I,O,U)$  of the B\&B tree, an upper bound $\UB(v)$ and a lower bound $\LB(v)$ on the optimal objective value of the corresponding subproblem are computed. If $\UB(v) \leq z^*$, the node is pruned, while if $\LB(v) > z^*$, the incumbent value is updated. In particular:
\begin{itemize}
    \item $\LB(v)$ is obtained by means of a constructive algorithm that completes the partial knapsack filling $I$ in a greedy manner. At each step of the algorithm, among all items $i \in U$ whose weight $w_i$ does not exceed the residual capacity $W - \sum_{j \in I} w_j$, the algorithm selects the item with the largest positive gain-to-weight ratio $g_i / w_i$, where the gain is defined as $g_i = \pi^*_i + \sum_{j \in I} p_{ij}$. This process is repeated until no remaining item has a positive gain or the residual capacity is insufficient to accommodate any of the remaining items;

    \item $\UB(v)$ is calculated as follows. First, for each item $i \in U$, an upper bound $\overline{p}_i$ on the maximum contribution obtainable by inserting item $i$ into the current partial solution $I$ is computed. Then, an upper bound on the maximum additional contribution achievable by extending $I$ with a pattern $P \in \mathcal{P}$ is obtained by solving $
    \max_{P' \in \mathcal{P}: I \subseteq P'}\set[\big]{ \sum_{j \in P' \cap U} \overline{p}_j}$,
    which added to the total profit value provided by $I$ gives the upper bound $\UB(v)$.

\end{itemize}

Let us describe in better detail how $\UB(v)$ is calculated. Given a pattern $P\in \mathcal{P}$, let $p(P) = \sum_{i \in P} \pi_i^*  + \sum_{i \in P} \sum_{j \in N : i < j} p_{ij}$ denote the total profit value provided by $P$.
Given an item $i \in U$, let $\bbnodekp{v}{i}$ be an instance of the KP01 with capacity $c_{v,i} = W - \sum_{j \in I} w_j - w_i$ 
and item set $U \setminus \set{i}$, each item $j$ with profit $p_{ij}^{+} = \max\set{ 0, p_{ij} }$ and weight $w_j$.
Let $z(\bbnodekp{v}{i})$ denote the optimal value of $\bbnodekp{v}{i}$.
Observe that $\overline{p}_i = \pi_i^* + \sum_{j \in I} p_{ij} + \frac{1}{2} z(\bbnodekp{v}{i})$ represents an upper bound on the maximum contribution that can be achieved by inserting item $i$ into $I$, which leads to the following.

\begin{proposition}\label{UB_QKP_part1}
Given a node $v = (I,O,U)$ of the B\&B tree and a pattern $P$ such that $I \subseteq P$, the inequality $p(P) \leq p(I) + \displaystyle\sum_{i \in P \cap U} \overline{p}_i$ holds.
\end{proposition}
\begin{proof}
Note that the total profit of $P$ can be decomposed as $p(P) = p(I)+p(P \cap U) + \displaystyle\sum_{i \in P \cap U} \sum_{j \in I} p_{ij}$. Moreover:
\begin{align*}
   p(P \cap U) + \sum_{i \in P \cap U} \sum_{j \in I} p_{ij} &=\sum_{i \in P \cap U}\paren*{\pi_i^* + \frac{1}{2} \sum_{j \in P \cap U\setminus\{i\}} p_{ij}} + \sum_{i \in P \cap U} \sum_{j \in I} p_{ij}= \\
        &=\sum_{i \in P \cap U}\paren*{\pi_i^* + \sum_{j \in I} p_{ij} + \frac{1}{2} \sum_{j \in P \cap U\setminus\{i\}} p_{ij}} \leq \\
        &\leq
       \sum_{i \in P \cap U}\paren*{\pi_i^* + \sum_{j \in I} p_{ij} + \frac{1}{2} \sum_{j \in P \cap U\setminus\{i\}} p^+_{ij}} \leq\\
        &\leq\sum_{i \in P \cap U}\paren*{\pi_i^* + \sum_{j \in I} p_{ij} + \frac{z(\bbnodekp{v}{i})}{2} }=\sum_{i \in P \cap U} \Bar{p}_i.
\end{align*}
\noindent
The first inequality follows from the fact that $p_{ij} \leq p^+_{ij}$, while the second one holds since $z(\bbnodekp{v}{i})=\displaystyle\max_{P'\in \mathcal{P}:I\subseteq P'}\hspace{-0.2cm}\sum_{j \in P'\cap U \setminus \{i\}} \hspace{-0.5cm} p_{ij}^+$. 
Combining this bound with the initial decomposition,  we obtain that
$p(P) \leq p(I)+\displaystyle\sum_{i \in P \cap U} \Bar{p}_i$, which concludes the proof.
\end{proof}

\begin{proposition}\label{UB_QKP}
 Given a node $v = (I,O,U)$ of the B\&B tree, the value $\UB(v)=p(I)+\displaystyle\max_{P'\in \mathcal{P}:I\subseteq P'}\hspace{-0.2cm}\sum_{j \in P'\cap U } \hspace{-0.2cm} \overline{p}_j$
 provides a valid upper bound on the optimal solution of the subproblem at node $v$.
\end{proposition}
\begin{proof}
The claim follows directly from the observation that $\UB(v)$ upper bounds the profit of any feasible pattern extending $I$.
Indeed, for any pattern $P$ such that $I \subseteq P$, Proposition \ref{UB_QKP_part1} implies that
$p(P) \leq p(I) + \displaystyle\sum_{i \in P \cap U} \overline{p}_i\leq p(I)+\displaystyle\max_{P'\in \mathcal{P}:I\subseteq P'}\hspace{-0.2cm}\sum_{j \in P'\cap U } \hspace{-0.2cm} \overline{p}_i=\UB(v)$.  
\end{proof}

Note that the computation of $\UB(v)$ requires solving a KP01 whose profit coefficients $\overline{p}_i$ are themselves obtained by solving the corresponding $\bbnodekp{v}{i}$ ($i \in U$), thus resulting in an overall time complexity of $\bigO(n^2 W)$.
To reduce the computational effort, we relax each problem $\bbnodekp{v}{i}$ ($i \in U$) by considering its fractional knapsack relaxation.
To this end, at the beginning of the B\&B algorithm, for each item $i \in N$ we sort the remaining items $j \in N \setminus \{i\}$ in non-increasing order of the ratio $p_{ij}^+ / w_j$. This preprocessing step requires a total time of $\bigO(n^2 \log n)$. As this ordering is independent of the specific node $v$, the fractional value of $z(KP_{v,i})$, and hence $\overline{p}_i$, can be computed during the execution of B\&B in $\bigO(n)$ time. The resulting time complexity for calculating $\UB(v)$ is therefore $\mathcal{O}(n^2 + nW)$.

\paragraph{Branching Rule and Search Exploration of the B\&B.}
Let $\mathcal{V}$ denote the set of active nodes in the B\&B tree. 
Initially, since all items can potentially be packed into the knapsack, we set 
$\mathcal{V} := \{(\emptyset, \emptyset, N)\}$. 
At each iteration of the B\&B tree exploration, for each unsolved node 
$v=(I,O,U) \in \mathcal{V}$, we first move to $O$ all items in $U$ whose insertion into $I$ would exceed the knapsack capacity. 
If $U = \emptyset$, the node is pruned, and the incumbent solution is updated if $z^* < p(I)$.  
Next, the bounding functions are applied to compute $\LB(v)$ and $\UB(v)$.  
If $\UB(v) \leq z^*$, the node is pruned.  
If $\LB(v) > z^*$, the incumbent solution is updated, and all nodes in $\mathcal{V}$ with an upper bound smaller than $z^*$ are pruned.
Among all nodes in $\mathcal{V}$, the node $v = (I,O,U)$ with the largest upper bound is selected for branching.  
For each item $j \in U$, we compute the upper bounds of the two corresponding child nodes, namely $ub_j^I = \UB(I \cup \{j\}, O, U \setminus \{j\})$ and $ub_j^O = \UB(I, O \cup \{j\}, U \setminus \{j\})$.  
The item chosen for branching is then $j' \in \arg\min_{j \in U} \{\max(ub_j^I, ub_j^O)\}$.  
The two resulting child nodes $(I \cup \{j'\}, O, U \setminus \{j'\})$ and $(I, O \cup \{j'\}, U \setminus \{j'\})$ are added to $\mathcal{V}$, the node $v$ is removed from $\mathcal{V}$, and the exploration of the B\&B tree continues to the next iteration until $\mathcal{V} = \emptyset$.
% \end{itemize}

\subsection{Branching Rule and Search Exploration of the B\&P}
\label{sec:branching_rules}

Let $\Pi$ denote the set of active nodes of the B\&P tree at a given iteration. 
Each node of the search tree is associated with a specific RMP, its optimal solution $\lambda^{*}$, the corresponding lower bound, and a conflict graph~$G_c$.

At the first iteration, we initialize $\Pi := \{\mathrm{RMP}\}$  and set the conflict graph associated with RMP to $G_c = (N,\emptyset)$. 
Then, at each iteration of the search, CG is applied to all nodes currently contained in $\Pi$ that have not yet been solved, in order to compute an optimal solution $\lambda^{*}$ of the corresponding RMP and its associated lower bound. 
The node in $\Pi$ with the smallest lower bound is then selected for branching. If the optimal solution $\lambda^{*}$ of the selected node is integer, then it is also optimal for the QBPP, and the search terminates. 
Otherwise, there exists a pair of items $i,j \in N$ such that 
$0 < \sum_{P \in \mathcal{P} : \{i,j\} \subseteq P} \lambda^{*}_P < 1$.
In this case, the selected node is removed from $\Pi$, the pair of items $(i,j)$ with the most fractional value is identified, and two child nodes are generated following the branching scheme proposed by \citet{ryanfoster1981}:
a $0$-branch, corresponding to the decision of forbidding items $i$ and $j$ from being packed in the same bin, and a $1$-branch, corresponding to the decision of enforcing items $i$ and $j$ to be packed in the same bin.

The $0$-branch decision is imposed in the corresponding subproblem by inheriting from the RMP of the father node only the columns corresponding to patterns that do not contain both items $i$ and $j$.
Additionally, the conflict graph $G_c$ is updated by adding the pair $(i,j)$ to the conflict set of edges, and all the dissimilarity values in~\eqref{eq:QKP_formulation} associated with conflicting pairs in $G_c$  are set to infinity, thus preventing the addition of columns corresponding to patterns containing items in conflict during the CG process.

On the other hand, the $1$-branch decision is imposed in the corresponding subproblem by inheriting from the RMP of the father node either the columns in which items $i$ and $j$ do not occur, or the columns corresponding to patterns that contain both items $i$ and $j$ by merging them into a new item $k$. The new item is defined in the subproblem by setting $w_k = w_i + w_j$ and $d_{kl} = d_{il} + d_{jl}$ for all $l \in N \setminus \{i,j\}$. Additionally, the value of $d_{ij}$ is added to the lower and upper bounds obtained on its corresponding subproblem.

Lastly, the two child nodes generated by the branching process are inserted into $\Pi$, and the search proceeds to the next iteration.

\section{Computational Experiments}
\label{sec:experiments}
The algorithms presented in the previous sections and listed in Section~\ref{sec:algorithms_tested} were implemented in C++20 and executed on a single core of an Intel(R) Xeon(R) CPU E3-1245 v5 processor running at 3.5~GHz, with 32~GB of shared memory, under the Linux Ubuntu 22.04.3 LTS operating system. A time limit of 3600~seconds was imposed for each run.

Since, to the best of our knowledge, the QBPP is a novel problem, Section~\ref{sec:benchmarks} introduces a benchmark set, which is made publicly available to foster future research, while Section~\ref{sec:overall_results} presents and discusses the computational results obtained on this benchmark.

\subsection{Algorithms Tested}\label{sec:algorithms_tested}
The three compact MILP formulations, namely $\mathrm{F_{FGW}}$, $\mathrm{F_{2A}}$, and $\mathrm{F_{R}}$, together with their enhanced versions $\mathrm{F_{EFGW}}$, $\mathrm{F_{E2A}}$, and $\mathrm{F_{ER}}$, presented in Section~\ref{sec:mathematical_formulations}, were solved using CPLEX~22.1.0.
The resulting algorithms are denoted by FGW, 2A, and R, while their enhanced counterparts are denoted by EFGW, E2A, and ER, respectively.
Moreover, we refer to FGW$^{\text{+SB}}$ and 2A$^{\text{+SB}}$ as the implementations of the models $\mathrm{F_{FGW}}$ and $\mathrm{F_{2A}}$, respectively, together with their corresponding symmetry-breaking inequalities. Specifically, FGW$^{\text{+SB}}$ coincides with the implementation of the model having objective function
$\min \alpha \sum_{k \in N} x_k^k + \sum_{k \in N}\sum_{i \in N}\sum_{j \in N: i<j} d_{ij}\hat{x}_{ij}^k$,
subject to constraints \eqref{fqp_constr_partition}–\eqref{fqp_constr_binary_x} and \eqref{ffgw_constr_1}–\eqref{eq:constr_x_symm2}. Similarly, 2A$^{\text{+SB}}$ coincides with the model having objective function
$\min \alpha \sum_{k \in N} x_k^k + \sum_{i \in N}\sum_{j \in N: i<j} d_{ij} z_{ij}$,
subject to constraints \eqref{fqp_constr_partition}–\eqref{fqp_constr_binary_x}, \eqref{eq:constr_x_symm1}, \eqref{eq:constr_x_symm2}, and \eqref{fnc_constr_1}–\eqref{fnc_constr_bounds_z}.

Two versions of the B\&P algorithm of Section~\ref{sec:branch-and-price} were implemented to solve the set-partitioning formulation $\mathrm{F_{SP}}$, namely B\&P$^{(1,1)}$ and B\&P$^{(5,10)}$.
B\&P$^{(1,1)}$ serves as a baseline, with the MCH described in Section~\ref{sec:QKP_heuristic} configured with $k = 1$ and adding at most one column with the lowest negative reduced cost at each CG iteration.
In contrast, B\&P$^{(5,10)}$ sets $k = 5$ and adds up to 10 columns with the lowest negative reduced costs per CG iteration. These values were determined through preliminary tests aimed at finding the best trade-off between solution quality and computational effort.

\subsection{Benchmark Set}
\label{sec:benchmarks}

We generated a set of benchmark instances for the QBPP and made them publicly available at
\url{https://github.com/regor-unimore/Quadratic-Bin-Packing-Problem}.

The instance generation scheme is derived from the one proposed by~\citet{Hiley2006547} for the MQKP. Specifically, each generated QBPP instance is defined by a quadruple $(n, \mu, \delta, \sigma)$, where $n$ is the number of items, $\mu$ controls the relative impact of the two components of the QBPP objective function by defining the bin cost $\alpha$ and the bin capacity $W$ (as explained below), $\delta$ represents the density of nonzero pairwise dissimilarity terms, and $\sigma \in \{-, +, \pm\}$ specifies whether the dissimilarity values are constrained to be non-positive ($-$), non-negative ($+$), or unrestricted ($\pm$).

Similarly to~\citet{Hiley2006547}, item weights are generated independently according to the discrete uniform distribution $\mathcal{U}(1,50)$.
For the pairwise dissimilarity values, if $\sigma = +$ (resp., $\sigma = -$, $\sigma = \pm$), each pair of items is assigned a value independently drawn from $\mathcal{U}(0,100)$ (resp., $\mathcal{U}(-100,0)$, $\mathcal{U}(-50,50)$) with probability $\delta$, and set to zero with probability $1-\delta$.
The bin capacity $W$ is set to $20\%$ of the total weight of all items divided by $\mu$.
Assuming for the moment that $m = \left\lceil \sum_{i \in N} w_i / W \right\rceil$ bins are required to pack all items, let $\bar{n} = n / m$ denote the average number of items per bin. 
Moreover, $\bar{d} = \bar{n}\cdot\sum_{i,j \in N : i<j} d_{ij} / \binom{n}{2} $ represents the average dissimilarity among items within a bin, which is then used to set the bin cost in relation to $\bar{d}$ as $\alpha = \lceil \mu |\bar{d}| \rceil$.
Larger values of $\mu$ favor solutions using fewer bins, even if items are dissimilar, whereas smaller values favor grouping similar items together, possibly at the expense of using more bins. 

The parameter sets are defined as $n \in \{25,30,35,40,45\}$, $\mu \in \{0.6,1,2\}$, $\delta \in \{0.25,0.5,0.75\}$, and $\sigma \in \{-, +, \pm\}$.
All combinations of these parameters were considered, and five instances were generated for each configuration, resulting in a total of $675$ instances.
Specifically,  to better evaluate the impact of the parameter $\mu$, for each fixed combination of $(n,\delta,\sigma)$ we generated groups of 3 related instances that differ only in the value of $\mu$. Instances within the same group share the same set of items, i.e., identical item weights and pairwise dissimilarities, and differ only in the bin capacity $W$ and the bin cost $\alpha$.

Throughout the remainder of the paper, we consider three instance classes according to problem size: small-sized ($n=25$), small- to medium-sized ($n \in \{25,30,35\}$), and large-sized ($n \in \{40,45\}$).

\subsection{Computational Results}
\label{sec:overall_results}

Tables~\ref{table_FGW_2A} to \ref{table:large-instances} summarize the computational results, grouped by $\delta$, as this parameter has the least impact on the performance of the tested algorithms.
For the sake of completeness and reproducibility, all detailed instance-by-instance results are made publicly available at
\url{https://github.com/regor-unimore/Quadratic-Bin-Packing-Problem}. In these tables, “Opt” reports the total number of instances solved to proven optimality, “Sec” denotes the average solution time in seconds, and “\%Gap” indicates the average percentage gap, computed as $100 \cdot  (\UB{} - \LB{})/\vert\UB{}\vert$, where $\UB{}$ is the value of the best solution found by the algorithm and $\LB{}$ is the maximum lower bound provided by the same algorithm. 

Table~\ref{table_FGW_2A} reports the results obtained on the small-sized instances using the compact formulations FGW and 2A, both with and without symmetry-breaking inequalities, as well as their enhanced versions with strengthened continuous relaxations. Without symmetry breaking, FGW does not solve any instance to proven optimality within the time limit, while 2A solves only 4 instances. In contrast, the strengthened versions FGW$^{\text{+SB}}$ and 2A$^{\text{+SB}}$ significantly outperform their base counterparts, both in terms of solution time and number of instances solved to optimality, highlighting the effectiveness of symmetry-breaking inequalities. In particular, FGW$^{\text{+SB}}$ and 2A$^{\text{+SB}}$ solve 110 and 101 instances to proven optimality, respectively. The EFGW and E2A algorithms based on the enhanced formulations show a further improvement in performance, consistently solving more instances to optimality (122 and 111, respectively) with substantially lower solution times compared to both the base and the symmetry-breaking versions, confirming the benefit of strengthening the continuous relaxation. Table~\ref{table_R} shows that, while the base formulation R solves only 71 small-sized instances and has high runtime and gap, its enhanced version ER has a better  performance, solving 106 instances with lower computational effort, but remains less effective than EFGW and E2A. A comparative analysis of the compact formulations’ performance is presented in Figure~\ref{fig:performance_compact_25}, which plots the solution time against the number of instances solved and further highlights the clear advantage of the enhanced formulations, in terms of both efficiency and number of instances solved to proven optimality.

\begin{table}[!]
\small
\caption{Computational results on small-sized instances for FGW, FGW$^{\text{+SB}}$, EFGW, 2A, 2A$^{\text{+SB}}$, and E2A.}\label{table_FGW_2A}
\centering
\setlength{\tabcolsep}{2.pt}
\resizebox{\columnwidth}{!}{%
\begin{tabular}{cccccrrcrrcrrcrrcrrcrrcrr}
\toprule
 & &  & 
& \multicolumn{3}{c}{FGW} 
& \multicolumn{3}{c}{FGW$^{\text{+SB}}$} 
& \multicolumn{3}{c}{EFGW} 
&   % colonna vuota
& \multicolumn{3}{c}{2A} 
& \multicolumn{3}{c}{2A$^{\text{+SB}}$} 
& \multicolumn{3}{c}{E2A} \\ 

\cmidrule(lr){5-7} \cmidrule(lr){8-10} \cmidrule(lr){11-13} 
\cmidrule(lr){15-17} \cmidrule(lr){18-20} \cmidrule(lr){21-23}
$n$ & $\sigma$ & $\mu$ & $\#$Tot
& Opt & Sec & \%Gap
& Opt & Sec & \%Gap
& Opt & Sec & \%Gap
&   % colonna vuota
& Opt & Sec & \%Gap
& Opt & Sec & \%Gap
& Opt & Sec & \%Gap \\
\cmidrule(lr){1-3}
\cmidrule(lr){4-4}
\cmidrule(lr){5-13}
\cmidrule(lr){15-23}
25&$\pm$&0.6&15&0&3600.0&66.9&\textbf{15}&109.2&0.0&\textbf{15}&180.5&0.0&&0&3600.0&86.4&14&840.8&0.4&\textbf{15}&619.0&0.0\\
25&$\pm$&1.0&15&0&3600.0&108.8&\textbf{15}&215.3&0.0&\textbf{15}&111.8&0.0&&0&3600.0&137.1&11&1466.4&5.3&\textbf{14}&618.9&0.6\\
25&$\pm$&2.0&15&0&3600.0&116.0&\textbf{15}&56.2&0.0&\textbf{15}&3.1&0.0&&0&3600.0&136.1&\textbf{15}&81.7&0.0&\textbf{15}&39.3&0.0\\
25&$+$&0.6&15&0&3600.0&41.4&\textbf{15}&272.3&0.0&\textbf{15}&277.9&0.0&&1&3561.0&40.0&\textbf{14}&298.9&0.3&\textbf{14}&309.3&0.3\\
25&$+$&1.0&15&0&3600.0&11.6&\textbf{12}&1460.8&1.0&11&1808.2&1.7&&1&3537.0&12.4&10&1609.2&2.3&\textbf{11}&1747.7&1.5\\
25&$+$&2.0&15&0&3600.0&7.2&13&971.0&0.8&\textbf{14}&323.5&0.5&&2&3120.1&3.5&\textbf{14}&349.3&0.5&\textbf{14}&361.1&0.5\\
25&$-$&0.6&15&0&3600.0&106.4&10&1691.1&6.7&\textbf{15}&770.7&0.0&&0&3600.0&120.8&8&2123.8&14.3&\textbf{13}&1463.8&0.6\\
25&$-$&1.0&15&0&3600.0&697.3&4&2937.4&501.8&\textbf{9}&2602.2&34.9&&0&3600.0&702.7&\textbf{2}&3375.0&460.8&\textbf{2}&3368.0&291.3\\
25&$-$&2.0&15&0&3600.0&18.0&11&1308.2&1.5&\textbf{13}&546.7&0.7&&0&3600.0&14.5&\textbf{13}&1089.6&0.9&\textbf{13}&695.4&0.9\\
\cmidrule(lr){1-3} \cmidrule(lr){4-4} 
\cmidrule(lr){5-7} \cmidrule(lr){8-10} \cmidrule(lr){11-13} 
\cmidrule(lr){15-17} \cmidrule(lr){18-20} \cmidrule(lr){21-23}
\multicolumn{3}{l}{Avg./Total} &135
&0&3600.0&130.4
&110&1002.4&56.9
&\textbf{122}&736.1&4.2
&
&4&3535.4&139.3
&101&1248.3&53.9
&\textbf{111}&1024.7&32.9 \\
\bottomrule
\end{tabular}%
}
\end{table}

\begin{figure}[!htbp]
\centering
\begin{minipage}{0.48\textwidth}
    \centering
    \small
    \captionof{table}{Computational results on small-sized instances for R and ER.}\label{table_R}
    \setlength{\tabcolsep}{3pt}
    \resizebox{\columnwidth}{!}{%
    \begin{tabular}{cccccrrcrr}
    \toprule
     & &  & 
    & \multicolumn{3}{c}{R} 
    & \multicolumn{3}{c}{ER} \\ 
    \cmidrule(lr){5-7} \cmidrule(lr){8-10}
    $n$ & $\sigma$ & $\mu$ & $\#$Tot
    & Opt & Sec & \%Gap
    & Opt & Sec & \%Gap \\
    \cmidrule(lr){1-3} \cmidrule(lr){4-4}
    \cmidrule(lr){5-7} \cmidrule(lr){8-10} 
    25&$\pm$&0.6&15&\textbf{15}&12.3&0.0&\textbf{15}&17.5&0.0\\
    25&$\pm$&1.0&15&\textbf{15}&349.6&0.0&\textbf{15}&244.2&0.0\\
    25&$\pm$&2.0&15&10&1406.5&7.0&\textbf{15}&391.3&0.0\\
    25&$+$&0.6&15&10&1452.5&6.1&\textbf{14}&291.5&0.2\\
    25&$+$&1.0&15&0&3600.0&28.5&\textbf{10}&1626.2&1.8\\
    25&$+$&2.0&15&10&2317.2&10.2&\textbf{15}&487.3&0.0\\
    25&$-$&0.6&15&8&2143.4&6.4&\textbf{9}&1930.4&5.5\\
    25&$-$&1.0&15&0&3600.0&403.4&\textbf{2}&3508.8&286.6\\
    25&$-$&2.0&15&3&2968.0&17.9&\textbf{11}&1859.1&2.2\\
    \cmidrule(lr){1-3} \cmidrule(lr){4-4}
    \cmidrule(lr){5-7} \cmidrule(lr){8-10} 
    \multicolumn{3}{l}{Avg./Total}&135&71&1983.3&53.3&\textbf{106}&1150.7&32.9\\
    \bottomrule
    \end{tabular}%
    }
\end{minipage}
\hfill
\begin{minipage}{0.48\textwidth}
    \centering
    \vspace{0.4cm} % spazio verticale per allineare con la tabella
    \captionof{figure}{Overall performance  on small-sized instances for the compact formulations.}
    \label{fig:performance_compact_25}
    % \resizebox{\textwidth}{!}{\input{fig/performance_compact_25.pgf}}
    \includegraphics[width=\textwidth]{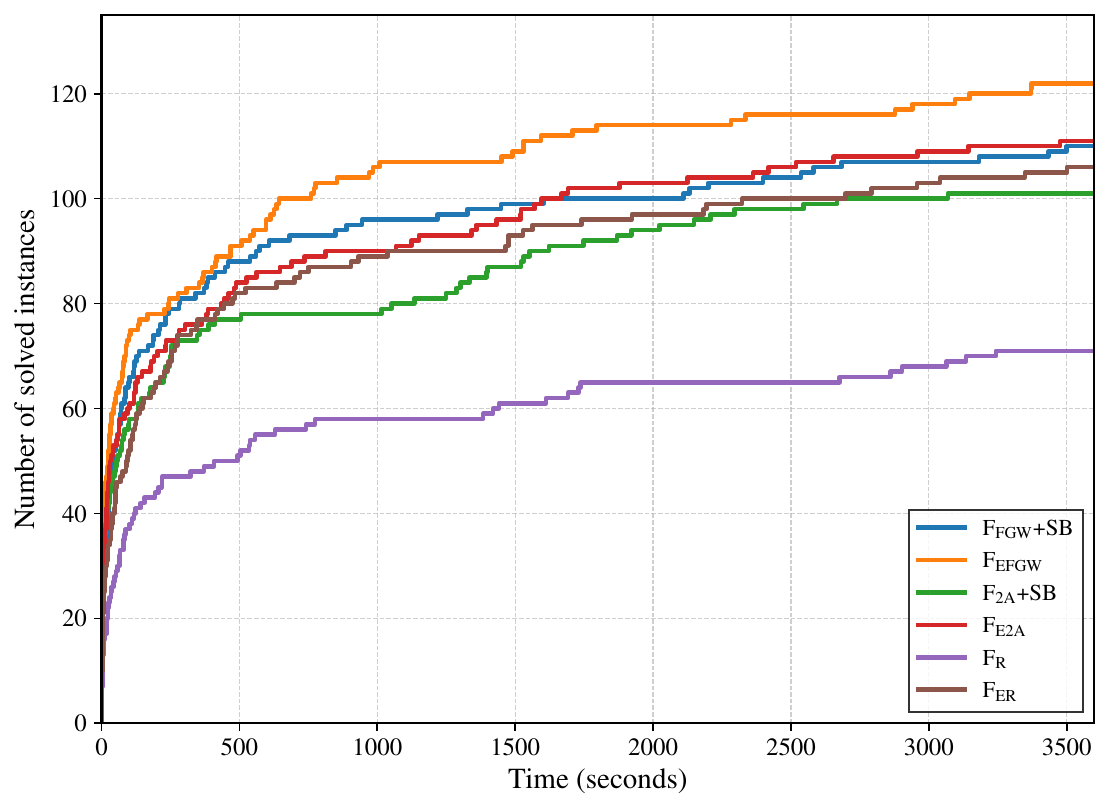}
\end{minipage}
\end{figure}
In this first round of experiments, a clear effect of the parameter $\sigma$ is also  observed. In particular, instances with $\sigma = -$ are consistently the most challenging ones across all tested algorithms, as indicated by the lower number of instances solved to proven optimality and the substantially higher average gaps and solution times. This is due to the presence of negative dissimilarity values in the objective function, which can drive the lower bound well below zero, increasing problem difficulty and potentially resulting in gaps larger than 100\%. By contrast, for instances with $\sigma = +$, gaps cannot exceed 100\%, as the lower bound provided by the algorithms is always non-negative. Figure~\ref{Fig_sigma_25} provides a visual summary of the performance of the compact formulations on the small-sized instances across the three values of $\sigma$. The figure shows that the enhanced formulations (EFGW, E2A, and ER) consistently solve more instances to proven optimality across all three values of $\sigma$, with the difference being most pronounced for the most challenging cases (i.e., $\sigma = -$). In addition to the effect of $\sigma$, the parameter $\mu$ also influences problem difficulty. 
For a fixed value of $\sigma$, instances with $\mu = 1.0$ exhibit longer computational times and larger gaps. 
This behavior can be explained by the fact that, since $\mu$ controls the relative weight of the two components of the QBPP objective function, the case $\mu = 1.0$ corresponds to a more balanced contribution of the bin-usage and the average dissimilarity among items within a bin. 

Based on the results presented so far and the superior effectiveness of EFGW, E2A, and ER, the subsequent analysis focuses exclusively on these enhanced compact formulations for comparison with the proposed B\&P approaches.

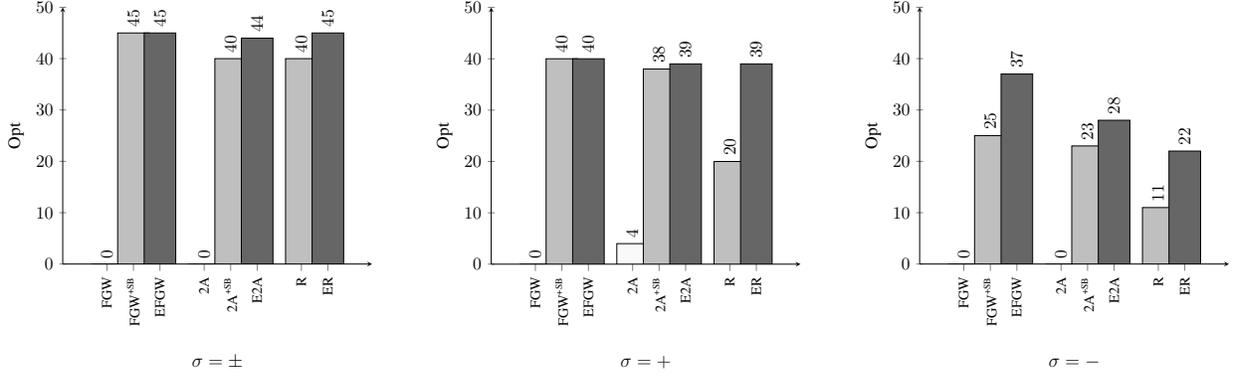
\begin{figure}[h!]
\centering
\begin{minipage}{0.31\textwidth}
    \centering
    \pgfkeys{/pgf/number format/.cd,1000 sep={}}
    \begin{tikzpicture}[scale=0.6]
\begin{axis}[
    xlabel={\large{$\sigma=\pm$}},
    xlabel style={yshift=-1em},
    ylabel={Opt},
    xtick={1, 2.5, 4, 6.5, 8, 9.5, 12, 13.5},
    ymax=50,
    xticklabels={\footnotesize{FGW}, \footnotesize{FGW$^{\text{+SB}}$},\footnotesize{EFGW}, \footnotesize{2A}, \footnotesize{2A$^{\text{+SB}}$},\footnotesize{E2A},\footnotesize{R},\footnotesize{ER}},
    x tick label style={rotate=90, anchor=east},
    bar width=0.7cm,
    nodes near coords, 
    nodes near coords align={vertical},
    nodes near coords style={rotate=90, anchor=west}, % <-- rotate numbers vertically
    axis x line=bottom,
    enlarge y limits=0.1,
    enlarge x limits=0.2,
    ybar=-0.7cm,
    ymin=0, 
    axis y line=left,
]

\addplot[ybar,fill=gray!5] coordinates {
(1,	0)
(6.5,	0)

};
\addplot[ybar,fill=gray!50] coordinates {
(2.5,	45)
(8,	40)
(12,	40)

};    
\addplot[ybar,fill=gray!120] coordinates {
(4, 45)
(9.5, 44)
(13.5, 45)
};

\end{axis}
\end{tikzpicture}
\end{minipage}\hfill
\begin{minipage}{0.31\textwidth}
    \centering
    \pgfkeys{/pgf/number format/.cd,1000 sep={}}
\begin{tikzpicture}[scale=0.6]
\begin{axis}[
    xlabel={\large{$\sigma=+$}},
    xlabel style={yshift=-1em},
    ylabel={Opt},
    xtick={1, 2.5, 4, 6.5, 8, 9.5, 12, 13.5},
    ymax=50,
    xticklabels={\footnotesize{FGW}, \footnotesize{FGW$^{\text{+SB}}$},\footnotesize{EFGW}, \footnotesize{2A}, \footnotesize{2A$^{\text{+SB}}$},\footnotesize{E2A},\footnotesize{R},\footnotesize{ER}},
    x tick label style={rotate=90, anchor=east},
    bar width=0.7cm,
    nodes near coords, 
    nodes near coords align={vertical},
    nodes near coords style={rotate=90, anchor=west}, % <-- rotate numbers vertically
    axis x line=bottom,
    enlarge y limits=0.1,
    enlarge x limits=0.2,
    ybar=-0.7cm,
    ymin=0, 
    axis y line=left,
]

\addplot[ybar,fill=gray!5] coordinates {
(1,	0)
(6.5,	4)
};
\addplot[ybar,fill=gray!50] coordinates {
(2.5,	40)
(8,	38)
(12,	20)

};    
\addplot[ybar,fill=gray!120] coordinates {
(4, 40)
(9.5, 39)
(13.5, 39)

};

\end{axis}
\end{tikzpicture}

\end{minipage}\hfill
\begin{minipage}{0.31\textwidth}
    \centering
    \pgfkeys{/pgf/number format/.cd,1000 sep={}}
\begin{tikzpicture}[scale=0.6]
\begin{axis}[
    xlabel={\large{$\sigma=-$}},
    xlabel style={yshift=-1em},
    ylabel={Opt},
    xtick={1, 2.5, 4, 6.5, 8, 9.5, 12, 13.5},
    ymax=50,
    xticklabels={\footnotesize{FGW}, \footnotesize{FGW$^{\text{+SB}}$},\footnotesize{EFGW}, \footnotesize{2A}, \footnotesize{2A$^{\text{+SB}}$},\footnotesize{E2A},\footnotesize{R},\footnotesize{ER}},
    x tick label style={rotate=90, anchor=east},
    bar width=0.7cm,
    nodes near coords, 
    nodes near coords align={vertical},
    nodes near coords style={rotate=90, anchor=west}, % <-- rotate numbers vertically
    axis x line=bottom,
    enlarge y limits=0.1,
    enlarge x limits=0.2,
    ybar=-0.7cm,
    ymin=0, 
    axis y line=left,
]

\addplot[ybar,fill=gray!5] coordinates {
(1,	0)
(6.5,	0)};
\addplot[ybar,fill=gray!50] coordinates {
(2.5,	25)
(8,	23)
(12,	11)    
};    
\addplot[ybar,fill=gray!120] coordinates {
(4, 37)
(9.5, 28)
(13.5, 22)};

\end{axis}
\end{tikzpicture}
\end{minipage}
\caption{Number of small-sized instances solved to proven optimality by the implemented compact formulations across different values of $\sigma$.}
\label{Fig_sigma_25}
\end{figure}

Table~\ref{Table2} reports the computational results on the small- to medium-sized instances for the enhanced compact formulations and the two B\&P (i.e., B\&P$^{(1,1)}$ and B\&P$^{(5,10)}$). Among the enhanced compact formulations, EFGW performs best, achieving the lowest computational times and smallest gaps, and solving 222 out of 405 instances to proven optimality. It is followed by ER, which solves 185 instances, and E2A, which solves 172 instances.
For the enhanced compact formulations, the trends regarding the impact of parameters $\sigma$ and $\mu$ observed on small-sized instances generally extend to the small- to medium-sized instances. However, as $n$ increases, the performance of these formulations clearly deteriorates: fewer instances with $n=35$ are solved to proven optimality, and both solution times and average gaps tend to increase. This underscores the sensitivity of these formulations to instance size and highlights their limited scalability.

By contrast, the B\&P algorithms exhibit robust performance across instance sizes. Both B\&P$^{(1,1)}$ and B\&P$^{(5,10)}$ solve nearly all small- to medium-sized instances, with very low computational times and small gaps. Instance size has only a minor impact on runtime, emphasizing the efficiency and scalability of the B\&P approach. Table~\ref{Table2} shows that the proposed MCH heuristic enhances the performance of the B\&P algorithm: B\&P$^{(5,10)}$ solves 389 out of 405 instances, compared to 386 solved by B\&P$^{(1,1)}$, while reducing the average runtime from 318.2 to 281.4 seconds and lowering the average optimality gap from 0.2\% to 0.1\%.
Figure~\ref{Fig_sigma_25_35} summarizes the performance of the compact formulations and the B\&P approaches in terms of the number of small- to medium-sized instances solved to proven optimality for the three considered values of $\sigma$. The figure illustrates that, unlike the compact formulations, the B\&P methods remain stable as $\sigma$ varies.

\begin{table}[!]
\small
\caption{Computational results for small- to medium-sized instances obtained with the enhanced compact formulations (EFGW, E2A, and ER) and the two B\&P implementations (B\&P$^{(1,1)}$ and B\&P$^{(5,10)}$), grouped by $\delta$.}
\centering
\setlength{\tabcolsep}{4pt}
\resizebox{\columnwidth}{!}{%
\begin{tabular}{cccrrrrrrrrrrrrrrrr}
\toprule
& & & 
& \multicolumn{3}{c}{EFGW} 
& \multicolumn{3}{c}{E2A} 
& \multicolumn{3}{c}{ER}
& \multicolumn{3}{c}{ B\&P$^{(1,1)}$}
& \multicolumn{3}{c}{B\&P$^{(5,10)}$} \\ 
\cmidrule(lr){5-7} 
\cmidrule(lr){8-10} 
\cmidrule(lr){11-13}
\cmidrule(lr){14-16}
\cmidrule(lr){17-19}
$n$ & $\sigma$ & $\mu$ & $\#$Tot 
& Opt & Sec & Gap\% 
& Opt & Sec & Gap\%
& Opt & Sec & Gap\%
& Opt & Sec & Gap\%
& Opt & Sec & Gap\% \\ 
\cmidrule(lr){1-3} \cmidrule(lr){4-4}
\cmidrule(lr){5-7} \cmidrule(lr){8-10} 
\cmidrule(lr){11-13} \cmidrule(lr){14-16}
\cmidrule(lr){17-19}
25&$\pm$&0.6&15&\textbf{15}&180.5&0.0&\textbf{15}&619.0&0.0&\textbf{15}&17.5&0.0&\textbf{15}&1.2&0.0&\textbf{15}&0.8&0.0\\
25&$\pm$&1.0&15&\textbf{15}&111.8&0.0&14&618.9&0.6&\textbf{15}&244.2&0.0&\textbf{15}&0.7&0.0&\textbf{15}&0.5&0.0\\
25&$\pm$&2.0&15&\textbf{15}&3.1&0.0&\textbf{15}&39.3&0.0&\textbf{15}&391.3&0.0&\textbf{15}&2.2&0.0&\textbf{15}&2.2&0.0\\
25&+&0.6&15&\textbf{15}&277.9&0.0&14&309.3&0.3&14&291.5&0.2&\textbf{15}&0.6&0.0&\textbf{15}&0.4&0.0\\
25&+&1.0&15&11&1808.2&1.7&11&1747.7&1.5&10&1626.2&1.8&\textbf{15}&35.7&0.0&\textbf{15}&39.9&0.0\\
25&+&2.0&15&14&323.5&0.5&14&361.1&0.5&\textbf{15}&487.3&0.0&11&963.0&1.2&11&962.2&1.2\\
25&-&0.6&15&\textbf{15}&770.7&0.0&13&1463.8&0.6&9&1930.4&5.5&\textbf{15}&18.4&0.0&\textbf{15}&15.2&0.0\\
25&-&1.0&15&9&2602.2&34.9&2&3368.0&291.3&2&3508.8&286.6&\textbf{15}&27.6&0.0&\textbf{15}&23.5&0.0\\
25&-&2.0&15&\textbf{13}&546.7&0.7&13&695.4&0.9&11&1859.1&2.2&11&960.3&1.4&11&960.3&1.3\\

\cmidrule(lr){1-3} \cmidrule(lr){4-4}
\cmidrule(lr){5-7} \cmidrule(lr){8-10} 
\cmidrule(lr){11-13} \cmidrule(lr){14-16}
\cmidrule(lr){17-19}
\multicolumn{3}{l}{Avg./Total} &135&122&736.1&4.2&111&1024.7&32.9&106&1150.7&32.9&\textbf{127}&223.3&0.3&\textbf{127}&222.8&0.3\\
\midrule
30&$\pm$&0.6&15&10&1900.0&12.8&8&2222.9&18.0&\textbf{15}&377.1&0.0&\textbf{15}&6.8&0.0&\textbf{15}&6.3&0.0\\
30&$\pm$&1.0&15&11&1690.9&4.2&5&2671.3&16.8&10&1681.3&5.7&\textbf{15}&5.6&0.0&\textbf{15}&3.2&0.0\\
30&$\pm$&2.0&15&\textbf{15}&270.4&0.0&9&1996.6&9.6&3&3123.8&91.1&\textbf{15}&63.4&0.0&\textbf{15}&54.9&0.0\\
30&+&0.6&15&11&1633.4&1.6&10&1335.8&2.2&10&1306.2&2.5&\textbf{15}&4.7&0.0&\textbf{15}&3.1&0.0\\
30&+&1.0&15&0&3600.0&9.1&4&2898.7&9.5&3&2993.4&10.6&\textbf{13}&781.5&0.5&\textbf{13}&805.3&0.4\\
30&+&2.0&15&11&1607.9&0.4&3&3092.5&0.8&1&3441.5&16.0&\textbf{14}&245.5&0.3&\textbf{14}&247.8&0.3\\
30&-&0.6&15&5&2965.5&12.2&4&3317.0&15.5&5&2966.3&14.7&\textbf{15}&161.9&0.0&\textbf{15}&119.1&0.0\\
30&-&1.0&15&1&3570.7&159.5&0&3600.0&234.9&0&3600.0&810.3&14&517.1&0.7&\textbf{15}&398.1&0.0\\
30&-&2.0&15&11&1587.1&0.6&1&3370.7&1.8&0&3600.0&30.7&\textbf{15}&92.3&0.0&\textbf{15}&65.2&0.0\\

\cmidrule(lr){1-3} \cmidrule(lr){4-4}
\cmidrule(lr){5-7} \cmidrule(lr){8-10} 
\cmidrule(lr){11-13} \cmidrule(lr){14-16}
\cmidrule(lr){17-19}
\multicolumn{3}{l}{Avg./Total} &135&75&2091.8&22.3&44&2722.8&34.3&47&2565.5&109.1&131&208.7&0.2&\textbf{132}&189.2&0.1\\
\midrule
35&$\pm$&0.6&15&5&2762.4&48.5&4&3031.1&47.7&11&1410.4&3.1&\textbf{15}&33.6&0.0&\textbf{15}&30.1&0.0\\
35&$\pm$&1.0&15&5&2935.8&26.2&1&3535.2&40.9&6&2336.2&14.9&\textbf{15}&15.3&0.0&\textbf{15}&9.3&0.0\\
35&$\pm$&2.0&15&9&1610.7&8.0&2&3302.6&56.4&3&3262.2&142.2&\textbf{15}&64.2&0.0&\textbf{15}&40.6&0.0\\
35&+&0.6&15&4&3176.9&12.1&10&2020.5&7.2&10&1938.6&7.3&\textbf{15}&85.4&0.0&\textbf{15}&67.3&0.0\\
35&+&1.0&15&0&3600.0&14.5&0&3600.0&15.5&0&3600.0&24.4&\textbf{15}&1066.1&0.0&\textbf{15}&981.7&0.0\\
35&+&2.0&15&1&3368.5&0.7&0&3600.0&1.1&0&3600.0&35.9&\textbf{12}&1050.8&0.1&11&1121.6&0.1\\
35&-&0.6&15&1&3514.2&35.4&0&3600.0&37.1&2&3325.4&25.1&13&1062.3&0.2&\textbf{14}&809.4&0.2\\
35&-&1.0&15&0&3600.0&113.2&0&3600.0&119.2&0&3600.0&250.0&13&1310.4&0.3&\textbf{15}&822.3&0.0\\
35&-&2.0&15&0&3600.0&3.6&0&3600.0&5.7&0&3600.0&59.0&\textbf{15}&15.6&0.0&\textbf{15}&7.1&0.0\\

\cmidrule(lr){1-3} \cmidrule(lr){4-4}
\cmidrule(lr){5-7} \cmidrule(lr){8-10} 
\cmidrule(lr){11-13} \cmidrule(lr){14-16}
\cmidrule(lr){17-19}
\multicolumn{3}{l}{Avg./Total} &135&25&3129.8&29.1&17&3321.0&36.7&32&2963.7&62.4&128&522.6&0.1&\textbf{130}&432.2&0.0\\
\midrule
\multicolumn{3}{l}{Overall} &405&222&1985.9&18.5&172&2356.2&34.6&185&2226.6&68.1&386&318.2&0.2&\textbf{389}&281.4&0.1\\
\bottomrule
\end{tabular}%
}
\label{Table2}
\end{table}

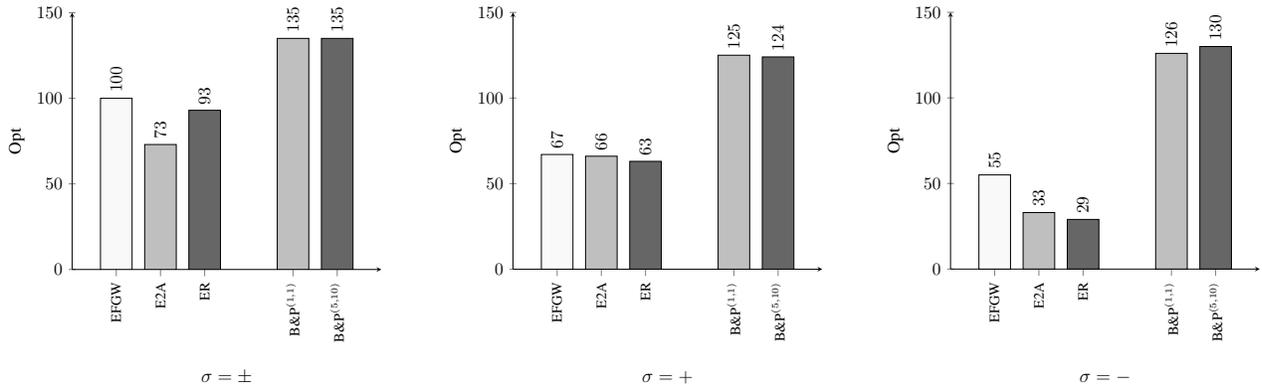
\begin{figure}[h!]
\centering
\begin{minipage}{0.31\textwidth}
    \centering
    \pgfkeys{/pgf/number format/.cd,1000 sep={}}

\begin{tikzpicture}[scale=0.6]
\begin{axis}[
    xlabel={\large{$\sigma=\pm$}},
    xlabel style={yshift=-1em},
    ylabel={Opt},
    xtick={1, 2, 3, 5, 6},
    ymax=150,
    xticklabels={\footnotesize{EFGW}, \footnotesize{E2A},\footnotesize{ER}, \footnotesize{B\&P$^{(1,1)}$},\footnotesize{B\&P$^{(5,10)}$}},
    x tick label style={rotate=90, anchor=east},
    bar width=0.7cm,
    nodes near coords, 
    nodes near coords align={vertical},
    nodes near coords style={rotate=90, anchor=west}, % <-- rotate numbers vertically
    axis x line=bottom,
    enlarge y limits=0.1,
    enlarge x limits=0.2,
    ybar=-0.7cm,
    ymin=0, 
    axis y line=left,
]

\addplot[ybar,fill=gray!5] coordinates {
(1,100)
};
\addplot[ybar,fill=gray!50] coordinates {
(2,73)
(5,135)
};    
\addplot[ybar,fill=gray!120] coordinates {
(6,135)
(3,93)};

\end{axis}
\end{tikzpicture}
\end{minipage}\hfill
\begin{minipage}{0.31\textwidth}
    \centering
    \pgfkeys{/pgf/number format/.cd,1000 sep={}}
\begin{tikzpicture}[scale=0.6]
\begin{axis}[
    xlabel={\large{$\sigma=+$}},
    xlabel style={yshift=-1em},
    ylabel={Opt},
    xtick={1, 2, 3, 5, 6},
    ymax=150,
    xticklabels={\footnotesize{EFGW}, \footnotesize{E2A},\footnotesize{ER}, \footnotesize{B\&P$^{(1,1)}$},\footnotesize{B\&P$^{(5,10)}$}},
    x tick label style={rotate=90, anchor=east},
    bar width=0.7cm,
    nodes near coords, 
    nodes near coords align={vertical},
    nodes near coords style={rotate=90, anchor=west}, % <-- rotate numbers vertically
    axis x line=bottom,
    enlarge y limits=0.1,
    enlarge x limits=0.2,
    ybar=-0.7cm,
    ymin=0, 
    axis y line=left,
]

\addplot[ybar,fill=gray!5] coordinates {
(1,67)

};
\addplot[ybar,fill=gray!50] coordinates {
(2,66)
(5,125)

};    
\addplot[ybar,fill=gray!120] coordinates {
(3,63)
(6,124)
};

\end{axis}
\end{tikzpicture}

\end{minipage}\hfill
\begin{minipage}{0.29\textwidth}
    \centering
    \pgfkeys{/pgf/number format/.cd,1000 sep={}}
\begin{tikzpicture}[scale=0.6]
\begin{axis}[
    xlabel={\large{$\sigma=-$}},
    xlabel style={yshift=-1em},
    ylabel={Opt},
    xtick={1, 2, 3, 5, 6},
    ymax=150,
    xticklabels={\footnotesize{EFGW}, \footnotesize{E2A},\footnotesize{ER}, \footnotesize{B\&P$^{(1,1)}$},\footnotesize{B\&P$^{(5,10)}$}},
    x tick label style={rotate=90, anchor=east},
    bar width=0.7cm,
    nodes near coords, 
    nodes near coords align={vertical},
    nodes near coords style={rotate=90, anchor=west}, % <-- rotate numbers vertically
    axis x line=bottom,
    enlarge y limits=0.1,
    enlarge x limits=0.2,
    ybar=-0.7cm,
    ymin=0, 
    axis y line=left,
]

\addplot[ybar,fill=gray!5] coordinates {
(1,55)

};
\addplot[ybar,fill=gray!50] coordinates {
(2,33)
(5,126)
};    
\addplot[ybar,fill=gray!120] coordinates {
(3,29)
(6,130)
};

\end{axis}
\end{tikzpicture}
\end{minipage}
\caption{Number of small- to medium-sized instances solved to proven optimality by the compact enhanced formulations and the two B\&P implementations, for different values of $\sigma$.}
\label{Fig_sigma_25_35}
\end{figure}

\cref{table:large-instances} reports the results for the large-sized instances.
Since the enhanced compact formulations already showed limited performance for instances with $n = 35$, we present results only for the B\&P algorithms.
In addition to the columns included in the previous tables, the column “\#Nodes” indicates the average number of nodes explored during the B\&P procedure.
To better assess the impact of the proposed MCH heuristic, we also report, for each column, the improvement achieved by B\&P$^{(5,10)}$ over B\&P$^{(1,1)}$, calculated as the difference between the corresponding values.
The results indicate that the inclusion of the MCH heuristic improves performance across nearly all instance types.
The employment of the MCH led to a worse performance in only three instance groups, slightly increasing the average computation time. Among these, the increased running time negatively affected the number of solved instances in only one group ($n=45$, $\sigma=-$, $\mu=0.6$), where one fewer instance was solved.
For the remaining instance groups, the improvements in terms of solved instances, runtime, optimality gap, and number of explored nodes were consistent. Overall, B\&P$^{(5,10)}$ solved nine additional instances and reduced the average computation time by 111 seconds, confirming the effectiveness of the proposed MCH heuristic.

\cref{Fig_sigma_40_45} illustrates the number of instances solved by the B\&P algorithms for each combination of $\sigma$ and $\mu$. For the large-sized instances, unlike the small-sized ones, the effect of $\mu$ on instance difficulty appears to be strongly dependent on the value of $\sigma$. Specifically, for instances with $\sigma = +$, the problem becomes harder as $\mu$ increases, whereas for $\sigma = -$, the opposite trend is observed, with difficulty increasing as $\mu$ decreases.
This behavior can be explained as follows. For instances with $\sigma = +$, optimal solutions tend to separate items as much as possible due to the dissimilarity values. As $\mu$ increases, the cost of using bins rises, strengthening the conflict between minimizing dissimilarity and minimizing bin cost. Consequently, balancing these competing objectives becomes increasingly challenging.
Conversely, for instances with $\sigma = -$, optimal solutions tend to group many items into the same bin. As $\mu$ decreases, the relative importance of bin cost reduces compared to the dissimilarity costs, and the bin capacity increases, leading to a larger number of feasible configurations in the set-partitioning formulation. Since obtaining strong dual bounds for the quadratic component is substantially more challenging than for the bin packing component, the combination of a more influential quadratic term and an increased number of feasible configurations further increases the overall difficulty of solving these instances.

\begin{table}[!]
\small
\caption{Computational results for large-sized instances obtained by B\&P$^{(1,1)}$ and B\&P$^{(5,10)}$, grouped by $\delta$.}
\centering
\setlength{\tabcolsep}{5.5pt}
\resizebox{\columnwidth}{!}{%
\begin{tabular}{cccrrrrrrrrrrrrr}
\toprule
& & & 
& \multicolumn{4}{c}{B\&P$^{(1,1)}$}
& \multicolumn{4}{c}{B\&P$^{(5,10)}$}
& \multicolumn{4}{c}{Improvements} \\ 
\cmidrule(lr){5-8} 
\cmidrule(lr){9-12}
\cmidrule(lr){13-16}
$n$ & $\sigma$ & $\mu$ & $\#$Tot
& Opt & Sec & \%Gap & \#Nodes
& Opt & Sec & \%Gap & \#Nodes
& Opt & Sec & \%Gap  & \#Nodes  \\ 
\cmidrule(lr){1-3} \cmidrule(lr){4-4}
\cmidrule(lr){5-8} \cmidrule(lr){9-12}
\cmidrule(lr){13-16}

40&$\pm$&0.6&15&\textbf{15}&440.7&0.0&4.7&\textbf{15}&287.4&0.0&4.1&0&-153.3&0.0&-0.7\\
40&$\pm$&1.0&15&\textbf{15}&254.3&0.0&26.6&\textbf{15}&171.0&0.0&25.3&0&-83.3&0.0&-1.3\\
40&$\pm$&2.0&15&\textbf{15}&261.5&0.0&1571.0&\textbf{15}&135.6&0.0&1639.1&0&-125.9&0.0&+68.1\\
40&$+$&0.6&15&\textbf{13}&596.7&0.7&1008.9&\textbf{13}&752.9&1.0&1844.9&0&+156.2&+0.3&+836.0\\
40&$+$&1.0&15&13&1428.7&1.8&3880.5&\textbf{14}&1250.4&1.0&4885.1&+1&-178.3&-0.8&+1004.6\\
40&$+$&2.0&15&\textbf{9}&1658.1&0.2&68284.5&\textbf{9}&1636.9&0.0&55197.1&0&-21.2&-0.2&-13087.3\\
40&$-$&0.6&15&4&2914.7&109.5&6.8&\textbf{5}&2833.5&1.0&7.9&+1&-81.3&-108.5&+1.1\\
40&$-$&1.0&15&5&2934.6&1.9&126.7&\textbf{10}&2538.5&2.0&259.2&+5&-396.1&+0.1&+132.5\\
40&$-$&2.0&15&\textbf{15}&324.0&0.0&997.8&\textbf{15}&154.3&0.0&1045.8&0&-169.7&0.0&+48.0\\

\cmidrule(lr){1-3} \cmidrule(lr){4-4}
\cmidrule(lr){5-8} \cmidrule(lr){9-12}
\cmidrule(lr){13-16}

\multicolumn{3}{l}{Avg./Total} 
&135&104&1201.5&12.7&8434.2&\textbf{111}&1084.5&0.6&7212.1&+7&-117.0&-12.1&-1222.1\\

\midrule
45&$\pm$&0.6&15&\textbf{10}&1775.4&19.3&2.3&\textbf{10}&1643.0&0.0&2.3&0&-132.4&-19.3&0.0\\
45&$\pm$&1.0&15&\textbf{15}&764.5&0.0&18.1&\textbf{15}&584.9&0.0&18.3&0&-179.6&0.0&+0.3\\
45&$\pm$&2.0&15&\textbf{15}&256.7&0.0&435.9&\textbf{15}&132.3&0.0&459.4&0&-124.4&0.0&+23.5\\
45&$+$&0.6&15&\textbf{12}&1283.3&1.0&1012.5&\textbf{12}&1303.6&1.0&1282.0&0&+20.3&+0.1&+269.5\\
45&$+$&1.0&15&\textbf{10}&1668.6&4.6&1129.5&\textbf{10}&1574.6&5.0&1260.5&0&-94.1&+0.4&+131.1\\
45&$+$&2.0&15&\textbf{8}&1873.6&0.2&43824.8&\textbf{8}&1854.6&0.0&37441.7&0&-19.0&-0.2&-6383.1\\
45&$-$&0.6&15&\textbf{2}&3541.0&344.8&2.1&1&3561.2&358.0&1.9&-1&+20.2&+13.2&-0.3\\
45&$-$&1.0&15&1&3395.7&9.6&11.9&\textbf{3}&3326.6&5.0&19.8&+2&-69.1&-4.6&+7.9\\
45&$-$&2.0&15&13&933.7&0.0&397.0&\textbf{14}&566.4&0.0&651.2&+1&-367.3&0.0&+254.2\\

\cmidrule(lr){1-3} \cmidrule(lr){4-4}
\cmidrule(lr){5-8} \cmidrule(lr){9-12}
\cmidrule(lr){13-16}

\multicolumn{3}{l}{Avg./Total} 
&135&86&1721.4&42.2&5203.8&\textbf{88}&1616.3&41.0&4570.8&+2&-105.0&-1.2&-633.0\\

\midrule
\multicolumn{3}{l}{Overall} 
&270&190&1461.4&27.4&6819.0&\textbf{199}&1350.4&20.8&5891.4&+9&-111.0&-6.6&-927.6\\

\bottomrule
\end{tabular}%
}
\label{table:large-instances}
\end{table}

\begin{figure}[h!]
\centering
\begin{minipage}{0.31\textwidth}
    \centering
    \pgfkeys{/pgf/number format/.cd,1000 sep={}}

\begin{tikzpicture}[scale=0.6]
\begin{axis}[
    xlabel={\large{$\sigma=\pm$}},
    xlabel style={yshift=-1em},
    ylabel={Opt},
    xtick={1, 2, 4, 5, 7,8},
    ymax=30,
    xticklabels={\footnotesize{B\&P$^{(1,1)}$},\footnotesize{B\&P$^{(5,10)}$},\footnotesize{B\&P$^{(1,1)}$},\footnotesize{B\&P$^{(5,10)}$}, \footnotesize{B\&P$^{(1,1)}$},\footnotesize{B\&P$^{(5,10)}$}},
    x tick label style={rotate=90, anchor=east},
    bar width=0.7cm,
    nodes near coords, 
    nodes near coords align={vertical},
    nodes near coords style={rotate=90, anchor=west}, % <-- rotate numbers vertically
    axis x line=bottom,
    enlarge y limits=0.1,
    enlarge x limits=0.2,
    ybar=-0.7cm,
    ymin=0, 
    axis y line=left,
]

\addplot[ybar,fill=gray!5] coordinates {
(1,	25)
(4,	30)
(7,	30)

};
\addplot[ybar,fill=gray!50] coordinates {
(2,	25)
(5,	30)
(8,	30)

};    

\addplot+[
    only marks, 
    mark size=5pt, 
    black, 
    nodes near coords, 
    point meta=explicit symbolic, 
    every node near coord/.append style={
        anchor=west,    % posizione del testo rispetto al marker
        font=\footnotesize,
        rotate=0       % ruota le scritte di 90°
    }
] 
coordinates {
    (2,2)[$\mu=0.6$]
    (1,2)[$\mu=0.6$]
    (4,2)[$\mu=1.0$]
    (5,2)[$\mu=1.0$]
    (7,2)[$\mu=2.0$]
    (8,2)[$\mu=2.0$]

};

\end{axis}
\end{tikzpicture}
\end{minipage}\hfill
\begin{minipage}{0.31\textwidth}
    \centering
    \pgfkeys{/pgf/number format/.cd,1000 sep={}}
\begin{tikzpicture}[scale=0.6]
\begin{axis}[
    xlabel={\large{$\sigma=+$}},
    xlabel style={yshift=-1em},
    ylabel={Opt},
    xtick={1, 2, 4, 5, 7,8},
    ymax=30,
    xticklabels={\footnotesize{B\&P$^{(1,1)}$},\footnotesize{B\&P$^{(5,10)}$},\footnotesize{B\&P$^{(1,1)}$},\footnotesize{B\&P$^{(5,10)}$}, \footnotesize{B\&P$^{(1,1)}$},\footnotesize{B\&P$^{(5,10)}$}},
    x tick label style={rotate=90, anchor=east},
    bar width=0.7cm,
    nodes near coords, 
    nodes near coords align={vertical},
    nodes near coords style={rotate=90, anchor=west}, % <-- rotate numbers vertically
    axis x line=bottom,
    enlarge y limits=0.1,
    enlarge x limits=0.2,
    ybar=-0.7cm,
    ymin=0, 
    axis y line=left,
]

\addplot[ybar,fill=gray!5] coordinates {
(1,	25)
(4,	23)
(7,	17)

};
\addplot[ybar,fill=gray!50] coordinates {
(2,	25)
(5,	24)
(8,	17)

};    

\addplot+[
    only marks, 
    mark size=5pt, 
    black, 
    nodes near coords, 
    point meta=explicit symbolic, 
    every node near coord/.append style={
        anchor=west,    % posizione del testo rispetto al marker
        font=\footnotesize,
        rotate=0       % ruota le scritte di 90°
    }
] 
coordinates {
    (2,2)[$\mu=0.6$]
    (1,2)[$\mu=0.6$]
    (4,2)[$\mu=1.0$]
    (5,2)[$\mu=1.0$]
    (7,2)[$\mu=2.0$]
    (8,2)[$\mu=2.0$]

};

\end{axis}
\end{tikzpicture}
\end{minipage}\hfill
\begin{minipage}{0.29\textwidth}
    \centering
    \pgfkeys{/pgf/number format/.cd,1000 sep={}}
\begin{tikzpicture}[scale=0.6]
\begin{axis}[
    xlabel={\large{$\sigma=-$}},
    xlabel style={yshift=-1em},
    ylabel={Opt},
    xtick={1, 2, 4, 5, 7,8},
    ymax=30,
    xticklabels={\footnotesize{B\&P$^{(1,1)}$},\footnotesize{B\&P$^{(5,10)}$},\footnotesize{B\&P$^{(1,1)}$},\footnotesize{B\&P$^{(5,10)}$}, \footnotesize{B\&P$^{(1,1)}$},\footnotesize{B\&P$^{(5,10)}$}},
    x tick label style={rotate=90, anchor=east},
    bar width=0.7cm,
    nodes near coords, 
    nodes near coords align={vertical},
    nodes near coords style={rotate=90, anchor=west}, % <-- rotate numbers vertically
    axis x line=bottom,
    enlarge y limits=0.1,
    enlarge x limits=0.2,
    ybar=-0.7cm,
    ymin=0, 
    axis y line=left,
]

\addplot[ybar,fill=gray!5] coordinates {
(1,	6)
(4,	6)
(7,	28)

};
\addplot[ybar,fill=gray!50] coordinates {
(2,	6)
(5,	13)
(8,	29)
};    

\addplot+[
    only marks, 
    mark size=5pt, 
    black, 
    nodes near coords, 
    point meta=explicit symbolic, 
    every node near coord/.append style={
        anchor=west,    % posizione del testo rispetto al marker
        font=\scriptsize,
        rotate=0       % ruota le scritte di 90°
    }
] 
coordinates {
    (2,0)[$\mu=0.6$]
    (1,0)[$\mu=0.6$]
    (4,0)[$\mu=1.0$]
    (5,0)[$\mu=1.0$]

};

\addplot+[
    only marks, 
    mark size=5pt, 
    black, 
    nodes near coords, 
    point meta=explicit symbolic, 
    every node near coord/.append style={
        anchor=west,    % posizione del testo rispetto al marker
        font=\footnotesize,
        rotate=0       % ruota le scritte di 90°
    }
] 
coordinates {
    (7,2)[$\mu=2.0$]
    (8,2)[$\mu=2.0$]

};

\end{axis}
\end{tikzpicture}
\end{minipage}
\caption{Comparison of the two B\&P implementations in terms of large-sized instances solved to proven optimality across $\sigma$ and $\mu$ values.}
\label{Fig_sigma_40_45}
\end{figure}
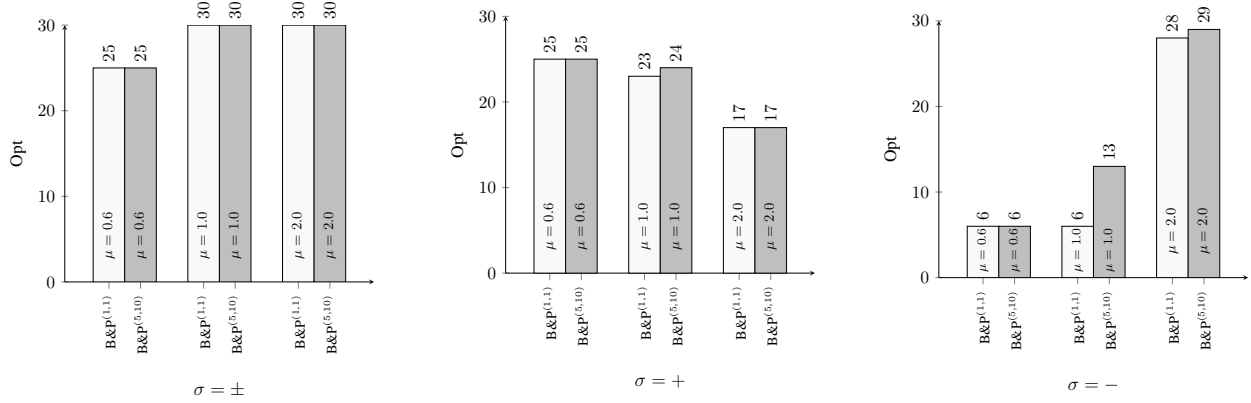

\section{Conclusions and Future Research Directions}
\label{sec:conclusion}

In this article, we introduced and analyzed the QBPP, a generalization of the classical BPP that incorporates a fixed cost for each used bin and pairwise costs (or profits) associated with items packed together. The QBPP is of both theoretical and practical significance, as it captures features relevant to real-world applications, such as those arising in cluster analysis. 

To address the QBPP, we proposed three compact MILP formulations, including strengthened variants, as well as a set-partitioning formulation solved via a tailored B\&P algorithm. The pricing problem arising in the column generation phase of the proposed B\&P algorithm corresponds to a variant of the QKP with potentially negative objective coefficients. To solve it efficiently, we developed a constructive heuristic capable of generating multiple high-quality solutions in a short time, complemented by a tailored exact combinatorial B\&B algorithm.

Given the novelty of the QBPP, we generated a benchmark set of instances, which we made publicly available along with the computational results to facilitate and encourage future research. Experiments on this benchmark show that while the compact formulations perform well on small instances, they scale poorly. In contrast, the B\&P algorithm provides a more robust performance across different instance types, efficiently solving larger instances and delivering stronger bounds, highlighting the clear advantage of tailored solution methods over general-purpose compact formulations, particularly as instance size increases. 
Moreover, the analysis of the computational results revealed that the hardness of the instances is strongly influenced by the interaction between the parameters $\sigma$ and $\mu$.
This observation suggests that different subclasses of instances exhibit clearly distinct structural properties. A deeper analysis of these characteristics, together with the design of solution methods tailored to specific combinations of $\sigma$ and $\mu$, could therefore yield further performance improvements.

Other future research directions include developing additional valid inequalities to further strengthen the proposed formulations and their separation procedures, exploring alternative approaches for solving the pricing problem (e.g., the $t$-linearization technique; see \citealt{Soares20212879}), and designing tailored column-stabilization methods to improve the convergence of CG within the B\&P framework (see, e.g., \citealt{CLAUTIAUX2025707}). Another promising avenue is the development of heuristic approaches, particularly for large instances that may arise in real-world applications. %Finally, extending the problem to accommodate item demands could further enhance both its theoretical and practical relevance.

\section*{Data availability}
The data used in this research are available at
\url{https://github.com/regor-unimore/Quadratic-Bin-Packing-Problem}.

\section*{Acknowledgments}

We gratefully acknowledge financial support from the European Union – NextGenerationEU, Mission 4 “Education and Research”, Component 1, Investment 3.4 “University Education and Advanced Skills”, Transnational Education Initiatives (TNE), Project “Green\&Pink for Sustainable Education” (CUP D74G23000280006); from the Emilia-Romagna regional funding program FSE+ 2021–2027 (Council Resolution No. 693/2023).
This study was financed, in part, by
the São Paulo Research Foundation (FAPESP), Brasil, proc.~\mbox{2022/05803-3},
the Coordenação de Aperfeiçoamento de Pessoal de Nível Superior - Brasil (CAPES) - Finance Code 001, proc.~\mbox{88882.329100/2014-01},
and the Conselho Nacional de Desenvolvimento Científico e Tecnológico (CNPq), proc.~\mbox{163645/2021-3}, \mbox{200503/2025-1}, \mbox{313146/2022-5}, \mbox{404315/2023-2}, \mbox{404481/2024-8} and \mbox{404779/2025-5}.

% References
\bibliographystyle{abbrvnat}
\bibliography{bib}

\end{document}